\documentclass[11pt]{article}
\usepackage{stmaryrd}
\usepackage{amsmath}
\usepackage{amsfonts}
\usepackage{amssymb}
\usepackage{amsthm}
\usepackage[all,2cell]{xy}
\usepackage{color}
\usepackage{etoolbox}

\usepackage[T2A,T1]{fontenc}
\usepackage[utf8]{inputenc}
\usepackage[russian,USenglish]{babel}

\UseTwocells

\newbool{hidedetails}
\newbool{erasevariant}
\booltrue{erasevariant}

\newtheorem{lemma}{Lemma}
\newtheorem{proposition}{Proposition}

\theoremstyle{definition}
\newtheorem{definition}{Definition}
\newtheorem{example}{Example}

\newcommand{\sets}{{\rm Set}}											
\newcommand{\op}{{\rm op}}											
\newcommand{\colim}{{\rm colim}}										
\newcommand{\id}{{\rm Id}}											
\newcommand{\projection}{{\pi}}										
\newcommand{\injection}{{\sigma}}										
\newcommand{\yoneda}{{\mathfrak Y}}									
\newcommand{\simplop}{{\Delta_\bullet^{\op}}}								

\newcommand{\closu}[1]{{\overline{#1}}}					

\newcommand{\comthe}{{\mathfrak R_{\rm alg}}}	
\newcommand{\ideal}{{\mathfrak A}}						
\newcommand{\igen}[1]{{(#1)}}							

\newcommand{\czero}{{\mathcal C}}										

\newcommand{\ceory}{{\mathfrak R_\infty}}								
\newcommand{\cinfty}{{C^\infty}}										
\newcommand{\allrings}{{\rm C^\infty R}}									
\newcommand{\embeddings}[1]{{\rm E(#1)}}								
\newcommand{\shpairs}{{{\rm E}(\Pre)}}								
\newcommand{\fgrings}{{\rm C^\infty R_{fg}}}								
\newcommand{\Pre}{{{\underline{\cschemes_{\rm fg}}}}}						
\newcommand{\Shea}{{\widetilde{\cschemes_{\rm fg}}}}					
\newcommand{\cschemes}{{\allrings^{\op}}}								
\newcommand{\fschemes}{{\fgrings^{\op}}}								
\newcommand{\cring}{{A}}												
\newcommand{\kernel}[1]{{{\rm Ker}(#1)}}									
\newcommand{\ctensor}{{\overline{\otimes}}}								

\newcommand{\D}{{$D$}}											

\newcommand{\nradical}[1]{{\sqrt[{\rm nil}]{#1}}}						
\newcommand{\iradical}[1]{{\sqrt[\infty]{#1}}}							
\newcommand{\Jradical}[1]{{\sqrt[{\rm J}]{#1}}}							
\newcommand{\pradical}[1]{{\sqrt[{\mathbb R\rm J}]{#1}}}					
\newcommand{\jradical}[1]{{\sqrt[{\rm j}]{#1}}}							
\newcommand{\gradical}[1]{{\sqrt[{\rm g}]{#1}}}							

\newcommand{\nil}{{\rm nil}}										
\newcommand{\redun}{{{\rm R}_{\nil}}}								
\newcommand{\redui}{{{\rm R}_\infty}}								
\newcommand{\redup}{{{\rm R}_{\rm pt}}}						 		
\newcommand{\nrings}{{\allrings_{\rm nil}}}							
\newcommand{\irings}{{\allrings_{\infty}}}								
\newcommand{\prings}{{\allrings_{\rm pt}}}							

\newcommand{\miderham}{{$\infty$- }}					
\newcommand{\infired}[1]{{#1_{\rm dR_\infty}}}					
\newcommand{\algered}[1]{{#1_{\rm dR}}}					

\newcommand{\diago}[1]{{\Delta}}										
\newcommand{\fornei}[1]{{(#1)^{\nil}}}									
\newcommand{\fornein}[2]{{\llbracket{#1}\rrbracket_{#2}^\tau}}			
\newcommand{\forneired}[2]{{\llbracket{#1}\rrbracket_{#2}^\tau}}				
\newcommand{\infinei}[1]{{(#1)^{\infty}}}									
\newcommand{\infinein}[2]{{(#1)_{#2}^\tau}}							
\newcommand{\infineired}[2]{{(#1)_{#2}^\tau}}								
\newcommand{\infinil}[1]{{\mathfrak I^{\infty}_{#1}}}			
\newcommand{\rinfinil}[2]{{\mathfrak I^{\infty}_{#1,#2}}}		
\newcommand{\fornil}[1]{{\mathfrak I^{\rm alg}_{#1}}}			
\newcommand{\rfornil}[2]{{\mathfrak I^{\rm alg}_{#1,#2}}}		
\newcommand{\charafu}[1]{{\Xi(#1)}}				

\newcommand{\zerovalue}[1]{{\mathfrak m_{#1}}}			
\newcommand{\zerojet}[1]{{\mathfrak m^\infty_{#1}}}			
\newcommand{\zerojetin}[2]{{\mathfrak m^\infty_{#1,#2}}}		
\newcommand{\zerogerm}[1]{{\mathfrak m^{\rm g}_{#1}}}		
\newcommand{\zerogermin}[2]{{\mathfrak m^{\rm g}_{#1,#2}}}	
\newcommand{\comsu}{{\mathfrak A_{\rm cs}}}				

\newcommand{\maxi}{{\mathfrak m}}						
\newcommand{\maxia}[1]{{\nradical{\mathfrak m_{#1}}}}					
\newcommand{\maxii}[1]{{\iradical{\mathfrak m_{#1}}}}						
\newcommand{\zeroset}[1]{{\mathfrak M(#1)}}				
\newcommand{\spec}{{\rm Spec}}						
\newcommand{\cspace}{{\mathcal X}}					
\newcommand{\cmanif}{{\mathcal M}}					
\newcommand{\Jet}[2]{{\llbracket{#1}\rrbracket_{#2}}}		
\newcommand{\Germ}[2]{{(#1)_{#2}}}					

\newcommand{\rzeroset}[1]{{\mathfrak M_\mathbb R(#1)}}	
\newcommand{\vset}[3][0]{{\{#3\,|\,#2(#3)=#1\}}}			
\newcommand{\rpt}{{\rm p}}							
\newcommand{\rqt}{{\rm q}}							
\newcommand{\rjet}[1]{{\rm J^\infty_{#1}}}					
\newcommand{\rgerm}[2]{{{#1}_{#2}}}						
\newcommand{\distance}[2]{{\rm{dist}(#1,#2)}}				
\newcommand{\upperrate}[4][]{{\beta^{#4}_{#1}(#3,#2)}}		
\newcommand{\ball}[2]{{B_{#1,#2}}}						
\newcommand{\decay}[3]{{#2\underset{#1}{\prec}#3}}			

\newcommand{\resh}[1]{{\underline{#1}}}								
\newcommand{\presh}{{\mathcal F}}									

\newcommand{\convo}[2]{{#1*#2}}						
\newcommand{\suppop}[1]{{{\rm supp}(#1)}}						

\newcommand{\blue}{\color{blue}}
\newcommand{\green}{\color{green}}

\newcommand{\hide}[1]{\ifbool{hidedetails}{}{{\blue #1}}}
\newcommand{\erase}[1]{\ifbool{erasevariant}{}{{\green #1}}}

\begin{document}

\title{Beyond perturbation 1: de Rham spaces}
\author{\small\parbox{.5\linewidth}{Dennis Borisov\\ University of Goettingen, Germany\\ Bunsenstr.\@ 3-4, 37073 G\"ottingen\\ dennis.borisov@gmail.com}
\parbox{.5\linewidth}{Kobi Kremnizer\\ University of Oxford, UK\\ Woodstock Rd, Oxford OX2 6GG\\ yakov.kremnitzer@maths.ox.ac.uk}}
\maketitle

\begin{abstract} It is shown that if one uses the notion of $\infty$-nilpotent elements due to Moerdijk and Reyes, instead of the usual definition of nilpotents to define reduced $\cinfty$-schemes, the resulting de Rham spaces are given as quotients by actions of germs of diagonals, instead of the formal neighbourhoods of the diagonals.

\smallskip



\end{abstract}

\tableofcontents

\section{Introduction}

As it is known (\cite{Gr68}), in algebraic geometry one can define connections, differential operators etc.\@ without ever mentioning derivatives or any kind of limiting procedure. One only needs to have the notion of infinitesimal neighbourhoods given by nilpotent elements. 

Looking at what happens when one contracts such neighbourhoods, one arrives at objects that have some prescribed behaviour along these contractions. The results of contractions are usually called de Rham spaces, and the objects that live on them are called crystals. Choosing to work with linear objects (i.e.\@ sheaves of modules) one arrives at \D-modules and linear differential operators (see e.g.\@ \cite{GR14}).

\smallskip

This technique can be applied also to differential geometry, once we use the theory of $\cinfty$-rings to take an algebraic-geometric approach (e.g.\@ \cite{MR91}). However, $\cinfty$-rings are much more than just commutative $\mathbb R$-algebras, and there is more than one way to define crystals in differential geometry because there is more than one notion of nilpotence. 

It was observed in \cite{MR86} that apart from the usual nilpotent elements, $\cinfty$-rings can have $\infty$-nilpotent ones, which are defined as follows: $a\in\cring$ is $\infty$-nilpotent, if the $\cinfty$-ring $\cring\{a^{-1}\}$ obtained by inverting $a$ is $0$. Here it is important that this inverting happens {\it in the category of $\cinfty$-rings}. For example $x\in\cinfty(\mathbb R)/(e^{-\frac{1}{x^2}})$ is $\infty$-nilpotent, but {\it not} nilpotent. Clearly every nilpotent element is also $\infty$-nilpotent.

Using the usual notion of nilpotence one gets de Rham spaces that can be described as quotients by actions of formal neighbourhoods of the diagonals. This is true in both algebraic and differential geometry. In the differential-geometric literature, however, one usually talks about jet bundles, instead of the formal neighbourhood of the diagonal and sheaves of modules over such neighbourhoods. The difference is in name only.

\smallskip

Something completely different happens when we apply de Rham space formalism to contracting $\infty$-nilpotent neighbourhoods. Instead of quotients by actions of formal neighbourhoods of diagonals we get quotients by actions of {\it germs} of diagonals. As in the formal case, these $\cinfty$-rings of germs carry a linear topology given by order of vanishing at the diagonal. Since in differential geometry there are many more orders of vanishing than just the finite ones, the infinitesimal theory we get here is much richer.

An immediate benefit of this richer theory is having many more differential operators, than just the polynomial ones. Another consequence becomes apparent when we consider the opposite procedure -- summation. Instead of the usual Taylor series, that characterize behavior of functions in relation to algebraic monomials, we need to work with functions having arbitrary vanishing properties, i.e.\@ we need to go to trans-series (e.g.\@ \cite{vanderHoeven06}) and beyond.

Consider the deformation theory one gets from this: in addition to specifying behavior of functors with respect to nilpotent extensions, we should consider also $\infty$-nilpotent extensions. Since all $\infty$-nilpotent extensions add up to germs, the deformation theory necessarily goes beyond perturbations. In turn this leads to a completely new definition of derived geometry. Differential graded manifolds or simplicial $\cinfty$-rings will not suffice anymore, since decomposition according to degree/simplicial dimension reflects decomposition according to finite orders of vanishing. 

The present paper is the first in a series examining this rich infinitesimal theory and possibly some of its applications.

\bigskip

Here is the contents of the paper:

\smallskip

In Section \ref{SixRadicals} we consider 6 different radicals of ideals in $\cinfty$-rings. Three of them are well known in commutative algebra (nilradical, Jacobson radical, and intersection of all maximal ideals having $\mathbb R$ as the residue field), while another two come from considering different Grothendieck topologies on the category of $\cinfty$-spaces. The main actor of this paper -- the $\infty$-radical -- is specific to the $\cinfty$-algebra.

In Section \ref{ThreeReductions} we prove that just as nilradical satisfies the strong functoriality property, so does the $\infty$-radical. We give proofs for some useful facts relating to algebra of these radicals, preparing the ground for de Rham spaces.

Before we can define de Rham spaces for both nilpotent and $\infty$-reductions we introduce regularity conditions in terms of injectivity with respect to the reduction functors. This is done in Section \ref{RegularityConditions}.

Then in Section \ref{DeRhamSpaces} we define the two kinds of de Rham spaces. Here we also prove the central result that these de Rham spaces can be built using certain neighbourhoods of the diagonals. The main ingredient in the proof is the strong functoriality from Section \ref{ThreeReductions}.

Section \ref{Reformulation} is dedicated to presenting these neighbourhoods of the diagonals as spectra of $\cinfty$-rings equipped with linear topologies. In the nilpotent case these spectra are just the usual formal neighbourhoods, while in the $\infty$-case they are the germs of diagonals.

De Rham groupoids appear in Section \ref{DeRhamGroupoids}. We show in particular that both in the nilpotent and the $\infty$- cases the de Rham spaces we construct are weakly equivalent (as simplicial sheaves) to the nerves of the corresponding de Rham groupoids. We show that often (e.g.\@ for all manifolds) these nerves consist of the formal neighbourhoods and respectively of the germs of diagonals in all cartesian powers.

Finally in Section \ref{DifferentialOperators} we look at the differential operators one obtains from $\infty$- de Rham groupoids. We postpone a detailed analysis to another paper, but we do show that these operators go beyond perturbation, i.e.\@ they provide infinitesimal description of deformations of the identity morphism on a manifold whose infinite jets at the identity vanish.

\bigskip

{\Large Acknowledgements}

\smallskip

Part of this work was done when the first author was visiting Max-Planck Institut f\"ur Mathematik in Bonn, several short visits to Oxford were also very helpful. Hospitality and financial support from both institutions are greatly appreciated. This work was presented in a series of talks in G\"ottingen, and the first author is grateful for the comments from C.Zhu and V.Pidstrygach.

\section{Radicals and reductions in $\cinfty$-algebra}

Let $\allrings$ be the category of $\cinfty$-rings, we denote by $\fgrings\subset\allrings$ the full subcategory of finitely generated $\cinfty$-rings. By definition (e.g.\@ \cite{MR91} \S I.1) $\allrings$ is the category of product preserving functors $\ceory\rightarrow\sets$, where $\ceory$ is the algebraic theory of $\cinfty$-rings, i.e.\@ it is the category having $\{\mathbb R^n\}_{n\in\mathbb Z_{\geq 0}}$ as objects and $\cinfty$-maps as morphisms. Objects of $\fgrings$ are quotients of $\{\cinfty(\mathbb R^n)\}_{n\geq 0}$ by ideals (\cite{MR91}, \S I.5). 

Geometric constructions are usually performed in $\fgrings$ (e.g.\@ \cite{MR91}), with infinitely generated $\cinfty$-rings recovered as ${\rm Ind}$-objects in $\fgrings$. However, sometimes it is necessary to consider infinitely generated $\cinfty$-rings directly. For example, there are many ideals in $\cinfty(\mathbb R^n)$, that are not finitely generated, consequently simplicial resolutions of $\cinfty$-rings often have components that are not finitely generated. This becomes important, for example, in \cite{QuasiCoherent}. In this paper we try to work in all of $\allrings$, and switch to $\fgrings$ only when necessary.

The category $\allrings$ is both complete and co-complete (e.g.\@ \cite{ARV11} Cor.\@ 1.22, Thm.\@ 4.5). We denote the coproduct in this category by $\ctensor$. For a $\cinfty$-ring $\cring$ we denote by $\spec(\cring)$ the corresponding object in the opposite category $\cschemes$. And given $\cspace\in\cschemes$ we write $\cinfty(\cspace)$ for the corresponding object of $\allrings$. For an arbitrary $\cspace\in\cschemes$ we denote by $\resh{\cspace}$ the representable pre-sheaf $\hom_{\cschemes}(-,\cspace)\colon\cschemes\rightarrow\sets$.

\subsection{Six radicals}\label{SixRadicals}

Let $\cring\in\allrings$ be a $\cinfty$-ring, as $\ceory$ contains the theory $\comthe$ of commutative, associative, unital $\mathbb R$-algebras, every $\ceory$-congruence on $\cring$ is given by an ideal of the underlying commutative $\mathbb R$-algebra. The converse is also true (e.g.\@ \cite{MR91}, prop.\@ I.1.2) and we can identify $\ceory$-congruences with ideals. There are several notions of radical ideals in this context, we consider $6$ of them. 

\begin{definition} Let $\ideal\leq\cring$ be an ideal\begin{itemize}
\item[1.] {\it the nilradical of $\ideal$} is $\nradical\ideal:=\{f\in\cring\,|\,\exists k\in\mathbb Z_{>0}\text{ s.t.\@ }f^k\in\ideal\}$,
\item[2.] {\it the $\infty$-radical of $\ideal$} is $\iradical\ideal:=\{f\in\cring\,|\,(\cring/\ideal)\{f^{-1}\}\cong \{0\}\}$,
\item[3.] {\it the Jacobson radical of $\ideal$} is $\Jradical{\ideal}:=\underset{\zeroset{\ideal}}\bigcap\maxi$, where $\zeroset{\ideal}$ is the set of maximal ideals containing $\ideal$,
\item[4.] {\it the $\mathbb R$-Jacobson radical of $\ideal$} is $\pradical{\ideal}:=\underset{\rzeroset{\ideal}}{\bigcap}\maxi$, where $\rzeroset{\ideal}\subseteq\zeroset{\ideal}$ consists of ideals with $\mathbb R$ as the residue field,
\item[5.] {\it the $\mathbb R$-jet radical of $\ideal$} is $\jradical{\ideal}:=\{f\in\cring\,|\,\forall\rpt\colon\cring\rightarrow\mathbb R\,\exists g\in\ideal\text{ s.t.\@ }\rjet{\rpt}(f)=\rjet{\rpt}(g)\}$,
where $\rjet{\rpt}(g)$ is the infinite jet of $g$ at $\rpt$,
\item[6.] {\it the $\mathbb R$-germ radical of $\ideal$} is $\gradical{\ideal}:=\{f\in\cring\,|\,\forall\rpt\colon\cring\rightarrow\mathbb R\,\exists g\in\ideal\text{ s.t.\@ }\rgerm{f}{\rpt}=\rgerm{g}{\rpt}\}$, where $\rgerm{g}{\rpt}$ is the germ of $g$ at $\rpt$.
\end{itemize}\end{definition}
The nilradical, Jacobson and $\mathbb R$-Jacobson radicals are borrowings from commutative algebra. The $\mathbb R$-jet and $\mathbb R$-germ radicals appear in defining various Grothendieck topologies on $\fschemes$ (e.g.\@ \cite{MR91} \S III) and natural topologies on $\cinfty$-rings (\cite{Bo15}). We will be mostly concerned with the $\infty$-radical, that was introduced in \cite{MR86} \S2. First we need to compare these radicals.

\begin{proposition}\label{FourInclusions} For any $\cring$, $\ideal$ as above there is a sequence of inclusions
	\begin{equation*}\nradical{\ideal}\subseteq\iradical{\ideal}\subseteq\Jradical{\ideal}\subseteq\pradical{\ideal},\end{equation*}
and each one of these inclusions can be strict.\end{proposition}
\begin{proof} The only inclusion that is not obvious is $\iradical{\ideal}\subseteq\Jradical{\ideal}$. Let $\alpha\in\cring$, and suppose there is $\maxi\in\zeroset{\ideal}$, s.t.\@ $\alpha\notin\maxi$. Then $[\alpha]\in\cring/\maxi$ is invertible, and hence $(\cring/\ideal)\{\alpha^{-1}\}$ cannot be trivial, i.e.\@ $\alpha\notin\iradical{\ideal}$.

\smallskip

In the following examples we put $\cring:=\cinfty(\mathbb R)$.\begin{itemize}

\item[1.] Let $\zerojet{0}\subset\cinfty(\mathbb R)$ consist of functions having $0$-jet at $0\in\mathbb R$. Then $\cinfty(\mathbb R)/\zerojet{0}\cong\mathbb R[[x]]$ (Borel lemma), and hence $\nradical{\zerojet{0}}=\zerojet{0}$. On the other hand the filter $\{V\subseteq\mathbb R\,|\,\exists f\in\zerojet{0},\, V=\vset{f}{\rpt}\}$ coincides with the filter of closed subsets of $\mathbb R$ containing $0$. Therefore from Lemma 2.2 in \cite{MR86} we conclude that $\iradical{\zerojet{0}}=\zerovalue{0}$, the ideal of functions vanishing at $0$. So $\nradical{\zerojet{0}}\subsetneq\iradical{\zerojet{0}}$.

\item[2.] Let $\zerogerm{0}\subset\cinfty(\mathbb R)$ consist of functions having $0$-germ at $0\in\mathbb R$. According to \cite{MR86}, page 329 the ring $\cinfty(\mathbb R)/\zerogerm{0}$ is local, and hence $\Jradical{\zerogerm{0}}=\zerovalue{0}$. On the other hand, using Lemma 2.2 from \cite{MR86} we see that $\iradical{\zerogerm{0}}=\zerogerm{0}$. So $\iradical{\zerogerm{0}}\subsetneq\Jradical{\zerogerm{0}}$.

\item[3.] Let $\comsu\subset\cinfty(\mathbb R)$ consist of functions having compact support. Clearly this is a proper ideal of $\cinfty(\mathbb R)$, and hence $\Jradical{\comsu}\neq\cinfty(\mathbb R)$. It is also clear that $\rzeroset{\comsu}=\emptyset$, i.e.\@ $\pradical{\comsu}=\cinfty(\mathbb R)$. So $\Jradical{\comsu}\subsetneq\pradical{\comsu}$.
\end{itemize}\end{proof}%

The other two radicals do not fit in this chain of inclusions. It is clear that we always have $\gradical{\ideal}\subseteq\jradical{\ideal}\subseteq\pradical{\ideal}$, and these inclusions can be strict (e.g.\@ applied to $\zerogerm{0}\subset\cinfty(\mathbb R)$). The following examples show that this is the most we can say in the general situation.

\begin{example}\label{NoFit} Consider $\igen{x^2}\subset\cinfty(\mathbb R)$. Clearly $\nradical{\igen{x^2}}=\zerovalue{0}$, while $\jradical{\igen{x^2}}=\igen{x^2}$. Thus $\jradical{\igen{x^2}}\subsetneq\nradical{\igen{x^2}}$.

Consider again the ideal $\comsu\subset\cinfty(\mathbb R)$ of functions with compact support. It is easy to see that $\gradical{\comsu}=\cinfty(\mathbb R)$, while as before $\Jradical{\comsu}\neq\cinfty(\mathbb R)$. Thus $\Jradical{\comsu}\subsetneq\gradical{\comsu}$.\end{example}
In summary we have the following relationship between the six radicals where each inclusion can be strict:
	\begin{equation}\nradical{\ideal}\subseteq\iradical{\ideal}\subseteq\Jradical{\ideal}\subseteq\pradical{\ideal}
	\supseteq\jradical{\ideal}\supseteq\gradical{\ideal}.\end{equation}
We finish this section with a useful characterization of \miderham\@radicals. Lemma 2.2 in \cite{MR86} gives an alternative to the definition: for any (finite or infinite) set $S$ $f\in\iradical{\ideal}\leq\cinfty(\mathbb R^S)$, iff $\exists g\in\ideal$, s.t.\@ $f$ and $g$ have the same sets of zeroes in $\mathbb R^S$. In addition to this reformulation we would like to have a generalization of the nilpotence condition. For any $\cinfty$-ring $\cring$ consider the $\cinfty$-ring $\cring\ctensor\cinfty(\mathbb R)$ obtained by freely adjoining a variable $x$. Substituting elements of $\cring$ instead of $x$ we have a morphism of $\cinfty$-rings
	\begin{equation*}\xymatrix{\cring\times(\cring\ctensor\cinfty(\mathbb R))\ar[r] & \cring,}\end{equation*}
which we denote by $(g,f)\mapsto f(g)$ for $f\in\cring\ctensor\cinfty(\mathbb R)$, $g\in\cring$. For example, choosing $f:=x^n\in\cring\ctensor\cinfty(\mathbb R)$ for some $n\in\mathbb N$, we have $f(g)=g^n$. In terms of spectra the operation $(g,f)\mapsto f(g)$ is the composition $\spec(\cring)\overset{\Gamma_g}\longrightarrow\spec(\cring)\times\mathbb R\overset{f}\longrightarrow\mathbb R$,
where $\Gamma_g$ is the graph of $g$.

\begin{definition} Let $\cring\in\allrings$, an element $g\in\cring$ is {\it $\infty$-nilpotent}, if there is $f\in\cring\ctensor\cinfty(\mathbb R)$, s.t.\@ $f(g)=0$ and $f$ becomes invertible in $\cring\ctensor\cinfty(\mathbb R\setminus\{0\})$.\end{definition}
This notion has an obvious extension to non-trivial ideals: let $\ideal\leq\cring$ be any ideal, {\it $g\in\cring$ is $\infty$-nilpotent relative to $\ideal$}, if there is $f\in\cring\ctensor\cinfty(\mathbb R)$, s.t.\@ $f(g)\in\ideal$ and $f$ becomes invertible in $(\cring/\ideal)\ctensor\cinfty(\mathbb R\setminus\{0\})$. However, it brings nothing new: let $\overline{\cring}:=\cring/\ideal$, denote the projection $\phi\colon\cring\rightarrow\overline{\cring}$, we have a commutative diagram in $\cschemes$:
	\begin{equation*}\xymatrix{\spec(\overline{\cring})\ar[rr]^{\Gamma_{\phi(g)}}\ar[d]_{\Phi} && 
	\spec(\overline{\cring})\times\mathbb R\ar[rr]^{\quad f\circ(\Phi\times\id_{\mathbb R})}\ar[d]_{\Phi\times\id_{\mathbb R}} && \mathbb R\ar@{-}[d]^=\\
	\spec(\cring)\ar[rr]_{\Gamma_{g}} && \spec(\cring)\times\mathbb R\ar[rr]_{\quad f} && \mathbb R,}\end{equation*}
where $\Phi$ corresponds to $\phi$. Then $g$ is $\infty$-nilpotent relative to $\ideal$, if $f\circ\Gamma_g\circ\Phi=0$. Since $\phi$ is surjective and $\cinfty(\mathbb R)$ is free, it is clear that this notion is equivalent to $\infty$-nilpotence of $\phi(g)$ in $\overline{\cring}$.

One can say that for such $g$ the ``degree of nilpotence'' is {\it at most} $f$, meaning the rate of vanishing of $f$ at $\spec(\cring)$. Notice that this ``degree of nilpotence'' is not only possibly infinite, but can also vary on $\spec(\cring)$.

\begin{proposition}\label{InfiniteNilpotence} Let $\cring\in\allrings$, and let $\ideal\leq\cring$ be any ideal. Then
	\begin{equation*}g\in\iradical{\ideal}\Longleftrightarrow\text{the class of }g\text{ in }\cring/\ideal\text{ is }\infty\text{-nilpotent}.\end{equation*}
\end{proposition}
\begin{proof} First we bring everything to $\cring/\ideal$ using the following simple lemma.
\begin{lemma}\label{SurjectiveFunctoriality} Let $\phi\colon\cring_1\rightarrow\cring_2$ be a surjective morphism in $\allrings$, and let $\ideal\leq\cring_2$ be an ideal. Then $\phi^{-1}(\iradical{\ideal})=\iradical{\phi^{-1}(\ideal)}$.\end{lemma}
\begin{proof} For any $a\in\cring_1$ consider the commutative triangle of $\cinfty$-morphisms
	\begin{equation*}\xymatrix{& \cring_1\ar[ld]\ar[rd] &\\ (\cring_1/\phi^{-1}(\ideal))\{a^{-1}\}\ar[rr] && (\cring_2/\ideal)\{(\phi(a)^{-1}\}.}\end{equation*}
Both arrows out of $\cring_1$ are initial among those $\cinfty$-morphisms that kill $\phi^{-1}(\ideal)$ and invert $a$. Therefore the horizontal arrow must be an isomorphism. Hence $a\in\iradical{\phi^{-1}(\ideal)}\Leftrightarrow\phi(a)\in\iradical{\ideal}\Leftrightarrow a\in\phi^{-1}(\iradical{\ideal})$.\end{proof}

\noindent Now we can assume $\ideal=0$. Let $x$ be a generator of $\cinfty(\mathbb R)$, $\forall g\in\cring$ we have
	\begin{equation*}\cring\{g^{-1}\}\cong((\cring\ctensor\cinfty(\mathbb R))/\igen{g-x})\{x^{-1}\}.\end{equation*}
Therefore $g\in\iradical{0}\Leftrightarrow x\in\iradical{\igen{g-x}}$. The latter is equivalent to existence of some $f\in\igen{g-x}$, that becomes invertible in $(\cring\ctensor\cinfty(\mathbb R))\{x^{-1}\}$ (\cite{MR86} page 329). Composing $\spec(\cring)\overset{\Gamma_g}\longrightarrow\spec(\cring)\times\mathbb R\overset{f}\longrightarrow\mathbb R$ we get $f(g)=0$.
Conversely, from $f(g)=0$ we get $f\in\igen{g-x}$, as $f(g)=0$ means $f$ has to be in the ideal of the graph of $g$, which is $\igen{g-x}$.\end{proof}%

\subsection{Three reductions}\label{ThreeReductions}

To define reductions we need to investigate functoriality properties of the radicals. We start with the following elementary lemma.

\begin{lemma}\label{WeakFunctoriality} Let $\phi\colon\cring_1\rightarrow\cring_2$ be a morphism in $\allrings$, and let $\ideal_1\leq\cring_1$, $\ideal_2\leq\cring_2$ be any ideals, s.t.\@ $\phi(\ideal_1)\subseteq\ideal_2$. Then
	\begin{equation*}\nradical{\ideal_1}\subseteq\phi^{-1}(\nradical{\ideal_2}),\quad\iradical{\ideal_1}\subseteq\phi^{-1}(\iradical{\ideal_2}),\quad
	\pradical{\ideal_1}\subseteq\phi^{-1}(\pradical{\ideal_2}).\end{equation*} 
\end{lemma}
\begin{proof} The nilradical of $\ideal$ is the intersection of prime ideals containing $\ideal$, and pre-images of prime-ideals are again prime ideals.

Maximal ideals with residue field $\mathbb R$ are kernels of surjective maps onto $\mathbb R$, and any $\cinfty$-morphism to $\mathbb R$ is surjective (elements of $\mathbb R$ are the structure constants of the theory of $\cinfty$-rings). Therefore pre-images of maximal ideals with residue field $\mathbb R$ are again such ideals.

According to \cite{MR86} page 329, $\alpha\in\iradical{\ideal_1}$,  iff $\exists\beta\in\ideal_1$, s.t.\@ $\beta$ becomes invertible in $\cring_1\{\alpha^{-1}\}$. Then $\phi(\beta)\in\ideal_2$, and we conclude that $\phi(\alpha)\in\iradical{\ideal_2}$.\end{proof}

In fact, for the nilradical and \miderham\@radical a much stronger result is true.

\begin{proposition}\label{StrongFunctoriality} Let $\phi\colon\cring_1\rightarrow\cring_2$ be a morphism in $\allrings$. Then
	\begin{equation*}\phi^{-1}(\nradical{0})=\nradical{\kernel{\phi}},\quad\phi^{-1}(\iradical{0})=\iradical{\kernel{\phi}}.\end{equation*}
\end{proposition}
\begin{proof} For the nilradical the statement is clear: $\phi(a)^k=0\Leftrightarrow a^k\in\kernel{\phi}$.

\smallskip

The case of $\infty$-radical. Let $\{a_i\}_{i\in I}$ be a set of generators of $\cring_1$ as a $\cinfty$-ring, and let $\{b_j\}_{j\in J}$ be a set of generators of $\cring_2$ obtained by enlarging $\{\phi(a_i)\}_{i\in I}$. Let $\psi\colon\cinfty(\mathbb R^I)\hookrightarrow\cinfty(\mathbb R^{J})$ be the $\cinfty$-morphism corresponding to $\Psi\colon\mathbb R^{J}\rightarrow\mathbb R^I$ given by $I\subseteq J$.\footnote{We say $f\in\cinfty(\mathbb R^n)$, if $f$ factors into a projection $\mathbb R^I\rightarrow\mathbb R^k$ and a smooth $\mathbb R^k\rightarrow\mathbb R$.}  Clearly $\phi$, $\psi$ make up a commutative diagram with the projections $\pi_1\colon\cinfty(\mathbb R^I)\rightarrow\cring_1$, $\pi_2\colon\cinfty(\mathbb R^J)\rightarrow\cring_2$. 

We claim it is enough to prove the proposition for $\psi$. Indeed, from Lemma \ref{WeakFunctoriality} we know that $\iradical{\kernel{\phi}}\leq\phi^{-1}(\iradical{0})$. This inclusion is an equality, iff $\pi_1^{-1}(\iradical{\kernel{\phi}})=\pi_1^{-1}(\phi^{-1}(\iradical{0}))$, but $\pi_1^{-1}(\phi^{-1}(\iradical{0}))=\psi^{-1}(\pi_2^{-1}(\iradical{0}))$, and from Lemma \ref{SurjectiveFunctoriality} we know that $\pi_1^{-1}(\iradical{\kernel{\phi}})=\iradical{\pi^{-1}(\kernel{\phi})}=\iradical{\psi^{-1}(\pi_2^{-1}(0))}$ and $\psi^{-1}(\pi_2^{-1}(\iradical{0}))=\psi^{-1}(\iradical{\pi_2^{-1}(0)})$.

 So it is enough to prove that $\forall\ideal\leq\cinfty(\mathbb R^{J})$ $\iradical{\psi^{-1}(\ideal)}=\psi^{-1}(\iradical{\ideal})$. From Lemma \ref{WeakFunctoriality} we already have one direction, it remains to show that	
	\begin{equation}\label{IdealsInAffine}\iradical{\psi^{-1}(\ideal)}\supseteq\psi^{-1}(\iradical{\ideal}).\end{equation}
For an $a\in\cinfty(\mathbb R^I)$ being in $\psi^{-1}(\iradical{\ideal})$ is equivalent to existence of $f\in\ideal$, s.t.\@ $\vset{\psi(a)}{\rpt}=\vset{f}{\rpt}\subseteq\mathbb R^J$ (\cite{MR86} Lemma 2.2). We can choose {\it finite} subsets $I'\subseteq I$, $J'\subseteq J$, s.t.\@ $I'\subseteq J'$ and $a\in\cinfty(\mathbb R^{I'})$, $f\in\cinfty(\mathbb R^{J'})$. We denote $m:=|I'|$, $m+n:=|J'|$, $\ideal':=\cinfty(\mathbb R^{m+n})\cap\ideal$.

The statement $a\in\iradical{\psi^{-1}(\ideal)}$ is equivalent to existence of $h\in\psi^{-1}(\ideal)$, s.t.\@ $\vset{a}{\rpt}=\vset{h}{\rpt}$. The following lemma implies then (\ref{IdealsInAffine}).

\begin{lemma}\label{RelativeOrders} Let $\Psi\colon\mathbb R^{m+n}\rightarrow\mathbb R^m$ be the projection. Let $V\subseteq\mathbb R^m$ be a closed subset. Let $f\in\cinfty(\mathbb R^{m+n})$ be such that $\vset{f}{\rpt}=\Psi^{-1}(V)$. There are $g\in\cinfty(\mathbb R^{m+n})$, $h\in\cinfty(\mathbb R^m)$, s.t.\@ $\psi(h)=f g$, and $\vset{g}{\rpt}=\Psi^{-1}(V)$.\end{lemma}
One proves this lemma by finding an $h\in\cinfty(\mathbb R^m)$, s.t.\@ $\vset{h}{\rpt}=V$ and $h\circ\Psi$ vanishes on $\Psi^{-1}(V)$ faster than $f$. This is easy given the classical theory of infinite orders (\cite{Ha1910}). Details are in Appendix \ref{ProofRelativeOrders}.\end{proof}%

A similar statement for the $\mathbb R$-Jacobson radical is wrong in general. Consider $\cinfty(\mathbb R)/\zerogerm{0}\rightarrow\cinfty(\mathbb R\setminus\{0\})/(\zerogerm{0})$, where $\zerogerm{0}\leq\cinfty(\mathbb R)$ consists of functions having $0$ germ at $0\in\mathbb R$. This morphism is injective because a smooth function in a punctured neighbourhood can be extended to the puncture in at most one way, yet $\pradical{0}$ in the domain is $\zerovalue{0}/\zerogerm{0}$, the ideal of functions that vanish at $0\in\mathbb R$, while in the codomain it is the entire ring.

\smallskip

The following lemma answers the natural question of whether by composing different radicals we get anything new. 

\begin{lemma}\label{Iteration} Let $\cring\in\allrings$, and let $\ideal\leq\cring$ be any ideal. Then
	\begin{equation*}\nradical{\nradical{\ideal}}=\nradical{\ideal},\quad\iradical{\iradical{\ideal}}=\iradical{\ideal},\quad
	\pradical{\pradical{\ideal}}=\pradical{\ideal},\end{equation*}
	\begin{equation*}\iradical{\nradical{\ideal}}=\iradical{\ideal},\quad\pradical{\iradical{\ideal}}=\pradical{\ideal},\quad
	\pradical{\nradical{\ideal}}=\pradical{\ideal},\end{equation*}
	\begin{equation*}\nradical{\iradical{\ideal}}=\iradical{\ideal},\quad\nradical{\pradical{\ideal}}=\pradical{\ideal},\quad
	\iradical{\pradical{\ideal}}=\pradical{\ideal}.\end{equation*}
\end{lemma}
\begin{proof} In the first row the only non-trivial statement is $\iradical{\iradical{\ideal}}=\iradical{\ideal}$. According to Proposition \ref{InfiniteNilpotence}, if $a\in\iradical{\iradical{\ideal}}$, there is $f\in\cring\ctensor\cinfty(\mathbb R)$, s.t.\@ $f$ becomes invertible in $\cring\ctensor\cinfty(\mathbb R\setminus\{0\})$ and $f(a)\in\iradical{\ideal}$. In turn this implies existence of $g\in\cring\ctensor\cinfty(\mathbb R)$ with the same invertibility properties, and s.t.\@ $g(f(a))\in\ideal$. Consider $f$ and $g$ as elements of $\cring\ctensor\cinfty(\mathbb R^2)$, via the two different inclusions $\cring\ctensor\cinfty(\mathbb R)\rightrightarrows\cring\ctensor\cinfty(\mathbb R^2)$. Then we see that $g(f(a))=(g(f))(a)$. Indeed, let $x,y$ be generators of $\cinfty(\mathbb R^2)$, then $g(f(a))$ is the class of $g$ modulo $\igen{a-x}+\igen{f-y}$, which is the same as $(g(f))(a))$ (reversing the order of division). Inverting $x$ implies inverting $f$, and then modulo $\igen{f-y}$ it means inverting $y$, which implies inverting $g$, thus $g(f)$ becomes invertible in $\cring\ctensor\cinfty(\mathbb R\setminus\{0\})$ and we conclude that $a\in\iradical{\ideal}$.

Each radical preserves the inclusion relation between ideals (Lemma \ref{WeakFunctoriality}), hence the other rows follow from the first because of Proposition \ref{FourInclusions}.\end{proof}%

\smallskip

Lemma \ref{WeakFunctoriality} implies that $\cring\mapsto\cring/\nradical{0}$, $\cring\mapsto\cring/\iradical{0}$, $\cring\mapsto\cring/\pradical{0}$ canonically extend to functors $\redun,\redui,\redup\colon\allrings\longrightarrow\allrings$.
\begin{definition} We will call these functors {\it the nilpotent, \miderham and point reductions}. If nilpotent, \miderham or point radical of $0$ is again $0$, the $\cinfty$-ring will be called respectively {\it reduced}, {\it \miderham\@reduced} and {\it point reduced}.\footnote{Finitely generated point reduced $\cinfty$-rings are easily identified with the point-determined rings from \cite{MR91}. However, Example \ref{NoFit} shows that reduced and \miderham\@reduced $\cinfty$-rings do not have similar descriptions in \cite{MR91}.} \end{definition}

The corresponding full subcategories of $\allrings$ consisting of reduced $\cinfty$-rings will be denoted by $\nrings$, $\irings$, $\prings$. It is clear (Proposition \ref{FourInclusions}) that we have a sequence of adjunctions
	\begin{equation*}\xymatrix{\allrings\ar@<.5ex>[rr]^{\redun} && \nrings\ar@<.5ex>[rr]^{\redui}\ar@<.5ex>[ll] && 
	\irings\ar@<.5ex>[rr]^{\redup}\ar@<.5ex>[ll] && \prings\ar@<.5ex>[ll],}\end{equation*}
with the right adjoints being inclusions of full subcategories. Lemma \ref{Iteration} tells us of course that $\redui\circ\redun=\redui$, $\redup\circ\redui=\redup$. We finish this section with the following simple lemma.

\begin{lemma}\label{Facts} Let $\cring\in\allrings$, and let $\ideal_1,\ideal_2\leq\cring$ be any ideals, then 
	\begin{equation*}\nradical{\ideal_1\cap\ideal_2}=\nradical{\ideal_1}\cap\nradical{\ideal_2},\quad
	\iradical{\ideal_1\cap\ideal_2}=\iradical{\ideal_1}\cap\iradical{\ideal_2}.\end{equation*}
Let $\cring_1,\cring_2\in\allrings$, and let $\ideal_1\leq\cring_1$, $\ideal_2\leq\cring_2$ be any ideals. Then
	\begin{equation*}\nradical{(\ideal_1)+(\ideal_2)}\geq(\nradical{\ideal_1})+(\nradical{\ideal_2}),\quad
	\iradical{(\ideal_1)+(\ideal_2)}\geq(\iradical{\ideal_1})+(\iradical{\ideal_2}),\end{equation*}
where $(\ideal_1)\leq\cring_1\ctensor\cring_2$ is generated by $\ideal_1\leq\cring_1\hookrightarrow\cring_1\ctensor\cring_2$ and so on.
\end{lemma}
\begin{proof} Since $\ideal_1\cap\ideal_2\leq\ideal_1$, $\ideal_1\cap\ideal_2\leq\ideal_2$, we have $\iradical{\ideal_1\cap\ideal_2}\leq\iradical{\ideal_1}$, $\iradical{\ideal_1\cap\ideal_2}\leq\iradical{\ideal_2}$ (Lemma \ref{WeakFunctoriality}). We choose a surjective $\alpha\colon\cinfty(\mathbb R^S)\rightarrow\cring$. Since $\alpha^{-1}(\iradical{-})=\iradical{\alpha^{-1}(-)}$ (Lemma \ref{SurjectiveFunctoriality}), we can assume $\cring=\cinfty(\mathbb R^S)$. According to \cite{MR86}, Lemma 2.2 $f\in\iradical{\ideal_1}\cap\iradical{\ideal_2}$, iff $\exists g_1\in\ideal_1$, $\exists g_2\in\ideal_2$, s.t.\@ $\vset{g_1}{\rpt}=\vset{f}{\rpt}=\vset{g_2}{\rpt}$. Then $g_1g_2\in\ideal_1\cap\ideal_2$ and $\vset{(g_1g_2)}{\rpt}=\vset{f}{\rpt}$, therefore $f\in\iradical{\ideal_1}\cap\iradical{\ideal_2}$ implies $f\in\iradical{\ideal_1\cap\ideal_2}$. The nilpotent case is well known.

The second statement is obvious in the nilpotent case. In case of $\iradical{-}$, choosing generators we can assume $\cring_1$, $\cring_2$ are free (Lemma \ref{SurjectiveFunctoriality}), and then the claim becomes obvious due to characterization of $f\in\iradical{\ideal}$ above.\end{proof}

\section{Two kinds of de Rham space}

\subsection{Regularity conditions}\label{RegularityConditions}

Usually an infinitesimal theory requires working with rings that satisfy some regularity conditions. Our constructions of de Rham spaces will also involve this type of assumptions, but they would be rather weak.

\begin{definition} Given a functor $\mathcal F\colon\allrings\rightarrow\allrings$, left adjoint to the inclusion $\mathcal F(\allrings)\subseteq\allrings$, we say that a $\cinfty$-ring $\cring$ is {\it $\mathcal F$-projective}, if for any $\cring'\in\allrings$ and any $\cinfty$-morphism $\cring\rightarrow\mathcal F(\cring')$ there is a lifting in
	\begin{equation*}\xymatrix{& \cring'\ar[d]^{u}\\ \cring\ar[r]\ar@{.>}[ru] & \mathcal F(\cring'),}\end{equation*}
where ${ u}$ is given by the unit of the adjunction.

We will say that $\cspace\in\cschemes$ is {\it $\redun$-injective} or {\it $\redui$-injective}, if $\cinfty(\cspace)$ is $\redun$- or $\redui$-projective respectively.
\end{definition}

\noindent Since the units for $\redun$, $\redui$ consist of surjective morphisms, it is clear that $\mathbb R^I$ is both $\redun$- and $\redui$-injective for any set $I$. Here are others.

\begin{example}\label{LocalizationProjective} Let $U\subseteq\mathbb R^n$ be an open subset ($n\in\mathbb Z_{\geq 0}$). It is known (e.g.\@ \cite{MR91}, Proposition I.1.6) that $\cinfty(U)\cong\cinfty(\mathbb R^n)\{f^{-1}\}$ for some $f$ in $\cinfty(\mathbb R^n)$. We claim that $\cinfty(U)$ is both $\redun$- and $\redui$-projective. Let $\cring\in\allrings$, let $\ideal\leq\Jradical{0}\leq\cring$ be any ideal contained in the Jacobson radical of $\cring$. Consider the following diagram of solid arrows
	\begin{equation*}\xymatrix{ && \cring\ar[d]\\ \cinfty(\mathbb R^n)\ar[r]\ar@{.>}[rru]^\psi & \cinfty(U)\ar[r]_\phi\ar@{.>}[ru] & \cring/\ideal,}\end{equation*}
with $\phi$ being any $\cinfty$-morphism. Since $\cinfty(\mathbb R^n)$ is free and $\cring\rightarrow\cring/\ideal$ is surjective, we have a lift $\psi$. As $f$ is invertible in $\cinfty(U)$, $\phi(f)$ does not belong to any maximal ideal of $\cring/\ideal$. Since $\ideal\leq\Jradical{0}$, the set of maximal ideals in $\cring$ coincides with the set of pre-images of maximal ideals in $\cring/\ideal$. Therefore any pre-mage of $\phi(f)$ does not belong to any maximal ideal of $\cring$, hence it is invertible. In particular $\psi(f)$ is invertible, and $\psi$ factors through $\cinfty(U)$. Uniqueness of $f^{-1}$ implies that the entire diagram is commutative. From Proposition \ref{FourInclusions} we know that $\nradical{0}\leq\Jradical{0}$, $\iradical{0}\leq\Jradical{0}$, and we conclude that $\cinfty(U)$ is both $\redun$- and $\redui$-projective.\end{example}

\begin{example}\label{ManifoldsInjective} Let $\cmanif\subseteq\mathbb R^n$ be a smooth manifold embedded in $\mathbb R^n$. It is known that $\cmanif$ is a retract of its tubular neighbourhood in $\mathbb R^n$, which is $\redun$- and $\redui$-injective (Example \ref{LocalizationProjective}). Hence any such manifold is $\redun$- and $\redui$-injective as well.\end{example}

\begin{example}\label{PointInjective} Let $\mathcal F\colon\allrings\rightarrow\allrings$ be a functor, left adjoint to the inclusion $\mathcal F(\allrings)\subseteq\allrings$, and let $\cring\in\allrings$, s.t.\@ $\mathcal F(\cring)\cong\mathbb R$. For any $\cring'\in\allrings$ any $\cinfty$-morphism $\cring\rightarrow\mathcal F(\cring')$ factors through $\mathcal F(\cring)\cong\mathbb R$. Since $\mathbb R$ is an initial object in $\allrings$, any $\mathbb R\rightarrow\mathcal F(\cring')$ lifts to $\mathbb R\rightarrow\cring'$. Thus $\cring$ is $\mathcal F$-projective.

For example: any $\cring\in\allrings$, that is finite dimensional as an $\mathbb R$-space, is both $\redun$- and $\redui$-projective. Another example: let $f\in\mathbb R^n$, such that $\vset{f}{\rpt}=\{0\}\subseteq\mathbb R^n$ and $\rjet{0}(f)=0$. As $\iradical{\igen{f}}=\zerovalue{0}$, the ideal of functions vanishing at $0\in\mathbb R^n$, we see that $\cinfty(\mathbb R^n)/\igen{f}$ is $\redui$-projective.\end{example}

\begin{example} Let $f:=e^{-\frac{1}{x^2}}\in\cinfty(\mathbb R)$. According to Example \ref{PointInjective} the $\cinfty$-ring $\cinfty(\mathbb R)/\igen{f}$ is $\redui$-projective. We claim that it is {\it not} $\redun$-projective. It is clear that $f\in\nradical{\igen{f^2}}$, and hence $\cinfty(\mathbb R)/\igen{f^2}\rightarrow\redun(\cinfty(\mathbb R)/\igen{f^2})$ factors through $\cinfty(\mathbb R)/\igen{f}$. On the other hand $\zerojet{0}\leq\cinfty(\mathbb R)$ (the ideal of functions with vanishing jet at $0\in\mathbb R$) is radical and we have a projection $\redun(\cinfty(\mathbb R)/\igen{f^2})\rightarrow\mathbb R[[x]]$. Altogether we have the following diagram
	\begin{equation}\xymatrix{& \cinfty(\mathbb R)/\igen{f}\ar[ld]\ar[d]^\pi\ar[rd]^= && \\
	\mathbb R[[x]] & \redun(\cinfty(\mathbb R)/\igen{f^2})\ar[l] & \cinfty(\mathbb R)/\igen{f}\ar[l] & \cinfty(\mathbb R)/\igen{f^2}.\ar[l]}\end{equation}
Suppose we had a $\cinfty$-morphism $\alpha\colon\cinfty(\mathbb R)/\igen{f}\rightarrow\cinfty(\mathbb R)/\igen{f^2}$ that lifts $\pi$. Since the composition with $\cinfty(\mathbb R)/\igen{f^2}\rightarrow\mathbb R[[x]]$ is the usual projection, we see that $\alpha([x])=[x+g]$, where $g\in\zerojet{0}$, with the brackets denoting classes modulo $\igen{f}$ and $\igen{f^2}$ respectively. Then we should have $f(x+g)\in\igen{f^2}$, yet
	\begin{equation*}\frac{f(x+g)}{f}=\frac{e^{-\frac{1}{(x+g)^2}}}{e^{-\frac{1}{x^2}}}=e^{-\frac{x^2-(x+g)^2}{(x+g)^2x^2}}=
	e^{\frac{2x g+g^2}{x^4+2x^3g+g^2x^2}},\end{equation*}
and since $g$ has $0$-jet at $0\in\mathbb R$, $x^4$ divides g and $\frac{f(x+g)}{f}\underset{x\rightarrow 0}\longrightarrow 1$. Note that it is important that the projection on $\mathbb R[[x]]$ remains constant, as the change of variables $x\mapsto\frac{x}{\sqrt{2}}$ maps $(f)$ onto $(f^2)$.\end{example}

\subsection{De Rham spaces as colimits of pre-sheaves}\label{DeRhamSpaces}

Recall that for an $\cspace\in\cschemes$ we denote by $\resh{\cspace}$ the pre-sheaf $\hom_{\cschemes}(-,\cspace)$. There are several Grothendieck topologies one might consider on $\cschemes$, we will mostly be working with representable pre-sheaves without mentioning a Grothendieck topology. Note, however, that not all of the common topologies on $\cschemes$ are sub-canonical (\cite{MR91}, \S III).

\begin{definition} Let $\presh$ be a pre-sheaf on $\cschemes$ with values in $\sets$. We denote by $\infired{\presh}$, $\algered{\presh}$ the composite functors $$\presh\circ\redui\colon\allrings\longrightarrow\sets,\quad\presh\circ\redun\colon\allrings\longrightarrow\sets.$$\end{definition}
Since $\redun$, $\redui$ are left adjoint to the corresponding inclusions they come with natural transformations $\id_{\allrings}\rightarrow\redui$, $\id_{\allrings}\rightarrow\redun$, and consequently there are canonical morphisms $\presh\rightarrow\infired{\presh}$, $\presh\rightarrow\algered{\presh}$. We would like to have explicit descriptions of $\infired{\resh{\cspace}}$ and $\algered{\resh{\cspace}}$ for respectively $\redui$- and $\redun$-injective $\cspace\in\cschemes$. To do this we need to start with special neighbourhoods of the diagonal.

\smallskip

Let $\diago{\cspace}\subseteq\cspace\times\cspace$ be the diagonal, and let $\zerovalue{\diago{\cspace}}$ be the kernel of $\cinfty(\cspace\times\cspace)\rightarrow\cinfty(\diago{\cspace})$. We define
	\begin{equation}\label{SetsOfIdeals}\fornil{\diago{\cspace}}:=\{\ideal\leq\cinfty(\cspace\times\cspace)\,|\,\nradical{\ideal}=
	\maxia{\diago{\cspace}}\},\end{equation}
	\begin{equation*}\infinil{\diago{\cspace}}:=\{\ideal\leq\cinfty(\cspace\times\cspace)\,|\,\iradical{\ideal}=\maxii{\diago{\cspace}}\}.\end{equation*}
Given $\ideal\in\fornil{\diago{\cspace}}$, we have the $\cinfty$-ring $\cinfty(\cspace\times\cspace)/\ideal$, and the relation of inclusion among elements of $\fornil{\diago{\cspace}}$ induces the structure of a category on this set of $\cinfty$-rings. Similarly for $\infinil{\diago{\cspace}}$. The following proposition implies that the corresponding categories of neighbourhoods of the diagonal are filtered.

\begin{proposition}\label{FilteredDiagrams} The sets of ideals $\fornil{\diago{\cspace}}$, $\infinil{\diago{\cspace}}$ are closed with respect to finite intersections.\end{proposition}
\begin{proof} Follows directly from Lemma \ref{Facts}.\end{proof}%

\begin{definition}\label{TwoNeighborhoods} For $\cspace\in\cschemes$, and $\fornil{\diago{\cspace}}$, $\infinil{\diago{\cspace}}$ as in (\ref{SetsOfIdeals}) we denote
	\begin{equation*}\fornei{\diago{\cspace}}:=\underset{\ideal\in\fornil{\diago{\cspace}}}\colim\;\resh{\spec(\cinfty(\cspace\times\cspace)/\ideal)},\quad\infinei{\diago{\cspace}}:=
	\underset{\ideal\in\infinil{\diago{\cspace}}}\colim\;\resh{\spec(\cinfty(\cspace\times\cspace)/\ideal)},\end{equation*}
with the colimits taken in the category of pre-sheaves on $\cschemes$.\end{definition}

Each $\spec(\cinfty(\cspace\times\cspace)/\ideal)$ is canonically embedded into $\cspace\times\cspace$, and these embeddings make $\cspace\times\cspace$ into a co-cone over the diagram given by the inclusions of ideals. Therefore there are canonical $\fornei{\diago{\cspace}}\rightarrow\resh{\cspace}\times\resh{\cspace}$, $\infinei{\diago{\cspace}}\rightarrow\resh{\cspace}\times\resh{\cspace}$, and composing them with the two projections on $\resh{\cspace}$, we obtain
	\begin{equation}\label{InfinityRelation}\xymatrix{\fornei{\diago{\cspace}}\ar@<.5ex>[r]\ar@<-.5ex>[r] & \resh{\cspace}, &
	\infinei{\diago{\cspace}}\ar@<.5ex>[r]\ar@<-.5ex>[r] & \resh{\cspace}.}\end{equation}
	
\begin{proposition}\label{DeRhamAsQuotient} Let $\cspace\in\cschemes$ be $\redun$-injective (respectively $\redui$-injective), the canonical morphism $\resh{\cspace}\rightarrow\algered{\resh{\cspace}}$ (respectively $\resh{\cspace}\rightarrow\infired{\resh{\cspace}}$) is a co-equalizer of $\fornei{\diago{\cspace}}\rightrightarrows\resh{\cspace}$ (respectively $\infinei{\diago{\cspace}}\rightrightarrows\resh{\cspace}$) from (\ref{InfinityRelation}).\end{proposition}
\begin{proof} We prove he \miderham case, the other one is similar. The reason everything works is strong functoriality of the corresponding radical (Proposition \ref{StrongFunctoriality}).

Since colimits in a category of pre-sheaves are computed object-wise, we need to show that $\forall\cspace'\in\cschemes$ the map $\hom_{\cschemes}(\cspace',\cspace)\rightarrow\infired{\resh{\cspace}}(\cspace')$ is a coequalizer of $\infinei{\diago{\cspace}}(\cspace')\rightrightarrows\hom_{\cschemes}(\cspace',\cspace)$. Again computing colimits object-wise we have
	\begin{equation*}\infinei{\diago{\cspace}}(\cspace')=
	(\underset{\ideal\in\infinil{\diago{\cspace}}}\colim\;\resh{\spec(\cinfty(\cspace\times\cspace)/\ideal}))(\cspace')\cong\end{equation*}
	\begin{equation*}\cong\underset{\ideal\in\infinil{\diago{\cspace}}}\colim\;
	\hom_{\allrings}(\cinfty(\cspace\times\cspace)/\ideal,\cinfty(\cspace')).\end{equation*}
As $\redui\colon\allrings\rightarrow\allrings$ is left adjoint, any $\cinfty(\cspace\times\cspace)/\ideal\rightarrow\redui(\cinfty(\cspace'))$ factors through $\redui(\cinfty(\cspace\times\cspace)/\ideal)=\redui(\cinfty(\diago{\cspace}))$. Therefore for any $\phi\colon\cinfty(\cspace\times\cspace)/\ideal\rightarrow\cinfty(\cspace')$ the two possible compositions in 
	\begin{equation*}\xymatrix{\cinfty(\cspace)\ar@<+.5ex>[r]\ar@<-.5ex>[r] & \cinfty(\cspace\times\cspace)/\ideal\ar[r]^{\quad\phi} &
	\cinfty(\cspace')\ar[r] & \redui(\cinfty(\cspace'))}\end{equation*}
are equal, i.e.\@ $\resh{\cspace}(\cspace')\rightarrow\infired{\resh{\cspace}}(\cspace')$ is a co-cone on $\infinei{\diago{\cspace}}(\cspace')\rightrightarrows\resh{\cspace}(\cspace')$. Next we show that this co-cone is universal. Since $\cspace$ is $\redui$-injective  
	\begin{equation*}\xymatrix{\resh{\cspace}(\cspace')=\hom_{\cschemes}(\cspace',\cspace)\ar[r] & 
	\hom_{\cschemes}(\redui(\cspace'),\cspace)=\infired{\resh{\cspace}}(\cspace')}\end{equation*} 
is surjective. Thus all we need to show is that $\forall \alpha,\beta\in\resh{\cspace}(\cspace')$, that become identified in $\infired{\resh{\cspace}}(\cspace')$, $\exists\gamma\in\infinei{\diago{\cspace}}(\cspace')$ which is projected to $\alpha$, $\beta$ by (\ref{InfinityRelation}).

Choose generators $\{x_i\}_{i\in I}$ for $\cinfty(\cspace)$. As $\alpha,\beta$ are identified in $\infired{\resh{\cspace}}(\cspace')$ we have $\forall i\in I$ $\alpha^*(x_i)-\beta^*(x_i)\in\iradical{0}\leq\cinfty(\cspace')$. Let $\{x^1_i,x^2_i\}_{i\in I}$ be the corresponding generators of $\cinfty(\cspace\times\cspace)$. Using Proposition \ref{StrongFunctoriality} we see that $\forall i\in I$ $x^1_i-x^2_i\in\iradical{\kernel{\alpha^*\ctensor\beta^*}}\leq\cinfty(\cspace\times\cspace)$, i.e.\@ $\iradical{\kernel{\alpha^*\ctensor\beta^*}}=\maxii{\diago{\cspace}}$ (the $\geq$ direction follows from $x^1_i-x^2_i\in\iradical{\kernel{\alpha^*\ctensor\beta^*}}$ and the opposite inequality holds because $\kernel{\alpha^*\ctensor\beta^*}\leq\maxii{\diago{\cspace}}$). Thus $\kernel{\alpha^*\ctensor\beta^*}\in\infinil{\diago{\cspace}}$, and putting $\gamma:=\alpha\times\beta$ we are done.\end{proof}%

\noindent If $\cinfty(\cspace)$ is finitely generated, $\infinei{\diago{\cspace}}, \fornei{\diago{\cspace}}$ have simpler descriptions.

\begin{proposition}\label{FinitePresentation} Suppose $\cinfty(\cspace)\in\allrings$ is finitely generated, then
	\begin{equation*}\fornei{\diago{\cspace}}\cong
	\underset{k\in\mathbb Z_{\geq 1}}\colim\;\resh{\spec(\cinfty(\cspace\times\cspace)/(\zerovalue{\diago{\cspace}})^k)},\end{equation*}
	\begin{equation*}\infinei{\diago{\cspace}}\cong
	\underset{f\in\charafu{\diago{\cspace}}}\colim\;\resh{\spec(\cinfty(\cspace\times\cspace)/\igen{f})},\end{equation*}
where $\charafu{\diago{\cspace}}:=\{f\in\cinfty(\cspace\times\cspace)\,|\,\iradical{\igen{f}}=\maxii{\diago{\cspace}}\}$. The colimits are taken in the category of pre-sheaves on $\cschemes$.\end{proposition}
\begin{proof} Let $\ideal\in\fornil{\diago{\cspace}}$, a standard argument shows that $\exists k\in\mathbb Z_{\geq 1}$, s.t.\@ $(\zerovalue{\diago{\cspace}})^k\leq\ideal$: let $\{x_i\}_{1\leq i\leq n}$ be generators of $\cinfty(\cspace)$, then $\{x_i^1-x_i^2\}_{1\leq i\leq n}$ generate $\zerovalue{\diago{\cspace}}$, where $x_i^1$, $x_i^2$ are the two pullbacks of $x_i$ to $\cspace\times\cspace$. By assumption $\nradical{\ideal}\geq\zerovalue{\diago{\cspace}}$, in particular there are $\{k_i\}_{1\leq i\leq n}\subset\mathbb Z_{\geq 1}$ s.t.\@ $\forall i\,(x_i^1-x_i^2)^{k_i}\in\ideal$. Putting $k:=n\max(k_1,\ldots,k_n)$ we are done. 
Similarly the following Lemma implies that for any $\ideal\in\infinil{\diago{\cspace}}$ there is $f\in\charafu{\diago{\cspace}}$, s.t.\@ $\igen{f}\leq\ideal$.
\begin{lemma} Let $\cring$ be a $\cinfty$-ring, and let $\ideal\leq\cring$ be a finitely generated ideal. For any $\ideal'\leq\cring$, s.t.\@ $\iradical{\ideal'}=\iradical{\ideal}$ there is $f\in\ideal'$, s.t.\@ $\iradical{\igen{f}}=\iradical{\ideal}$.\end{lemma}
\begin{proof} Choose a surjective morphism $\cinfty(\mathbb R^I)\rightarrow\cring$, where $I$ is some set. Let $\widetilde{\ideal}'$, $\widetilde{\ideal}$ be the pre-images of $\ideal'$, $\ideal$ in $\cinfty(\mathbb R^I)$. Let $\{f_1,\ldots,f_n\}$ be a set of generators of $\ideal$ as an ideal, and choose their pre-images $\{\widetilde{f}_1,\ldots,\widetilde{f}_n\}$ in $\widetilde{\ideal}$. From Proposition \ref{StrongFunctoriality} we know that $\iradical{\widetilde{\ideal}'}=\iradical{\widetilde{\ideal}}$, in particular $\forall i$ $\widetilde{f}_i\in\iradical{\widetilde{\ideal}'}$. This means there are $\{\widetilde{g}_1,\ldots,\widetilde{g}_n\}\subseteq\widetilde{\ideal}'$, s.t.\@ $\forall i$ $\widetilde{f}_i$ and $\widetilde{g}_i$ vanish at the same points. Let $f$ be the image in $\cring$ of $g:=\underset{1\leq i\leq n}\sum\;\widetilde{g}_i^2$. We claim that $\iradical{\igen{f}}=\iradical{\ideal}$. As $f\in\ideal'$ and $\iradical{\ideal'}=\iradical{\ideal}$, the $\leq$ direction is clear. Since $g$ vanishes exactly at the set of common zeroes of $\{\widetilde{f}_i\}_{1\leq i\leq n}$, it is clear that $\forall i$ $g\widetilde{f}_i$ and $\widetilde{f}_i$ vanish at the same points. Thus $\forall i$ $\widetilde{f}_i\in\iradical{\igen{g}}$. Consequently pre-image of $\iradical{\igen{f}}$ contains each $\widetilde{f}_i$, and hence the entire pre-image of $\ideal$. This implies $\ideal\leq\iradical{\igen{f}}$, and then $\iradical{\igen{f}}\geq\iradical{\ideal}$.\end{proof}%

We see that $\mathcal D'_{\rm alg}:=\{\resh{\spec(\cinfty(\cspace\times\cspace)/(\zerovalue{\diago{\cspace}})^k)}\}_{k\in\mathbb N}$ is a sub-category of $\mathcal D_{\rm alg}:=\{\resh{\spec(\cinfty(\cspace\times\cspace)/\ideal)}\}_{\ideal\in\fornil{\diago{\cspace}}}$, and every object in $\mathcal D_{\rm alg}$ is mapped to an object of $\mathcal D'_{\rm alg}$. Similarly for $\mathcal D'_\infty:=\{\resh{\spec(\cinfty(\cspace\times\cspace)/\igen{f})}\}_{f\in\charafu{\diago{\cspace}}}$ and $\mathcal D_\infty:=\{\resh{\spec(\cinfty(\cspace\times\cspace)/\ideal)}\}_{\ideal\in\infinil{\diago{\cspace}}}$. Each one of these categories is filtered (Proposition \ref{FilteredDiagrams}), therefore $\mathcal D'_{\rm alg}\subseteq\mathcal D_{\rm alg}$, $\mathcal D'_\infty\subseteq\mathcal D_\infty$ are right cofinal subcategories (\cite{Hirschhorn03}, Def.\@ 14.2.1), and hence colimits over $\mathcal D'_{\rm alg}$ and $\mathcal D'_\infty$ are isomorphic to colimits over $\mathcal D_{\rm alg}$ and respectively $\mathcal D_\infty$ (\cite{Hirschhorn03}, Thm.\@ 14.2.5).\end{proof}%

Suppose that $\cspace$ is a manifold, and hence so is $\cspace\times\cspace$. Proposition \ref{FinitePresentation} is very familiar in the algebraic case: $\fornei{\diago{\cspace}}$ is just the formal neighbourhood of the diagonal obtained by dividing by higher and higher powers of $\zerovalue{\diago{\cspace}}$. 

In the $\infty$-case the situation is more interesting: for a function $f$ in $\cinfty(\cspace\times\cspace)$ we have $\iradical{\igen{f}}=\zerovalue{\diago{\cspace}}$, if and only if $f$ vanishes exactly on the diagonal. As we will see in the following section this implies that $\infinei{\diago{\cspace}}$ is {\it the germ} of the diagonal.

\smallskip

We finish this section by relating the two kinds of de Rham spaces.

\begin{proposition}\label{Comparison}  For any $\cspace\in\cschemes$ there is a commutative diagram:
	\begin{equation*}\xymatrix{\fornei{\diago{\cspace}}\ar[rr]\ar[rd] && \infinei{\diago{\cspace}}\ar[ld]\\ & \resh{\cspace}\times\resh{\cspace}.}\end{equation*}
\end{proposition}
\begin{proof} Recall that $\nradical{\ideal}\leq\iradical{\ideal}$ for any ideal (Proposition \ref{FourInclusions}). This immediately implies $\fornil{\diago{\cspace}}\subseteq\infinil{\diago{\cspace}}$ and then the statement.\end{proof}%

\subsection{Reformulation in terms of topological rings}\label{Reformulation}

There is no need to always work with the entire $\resh{\cspace}\times\resh{\cspace}$. We will show here that the natural embeddings of $\fornei{\diago{\cspace}}$ and $\infinei{\diago{\cspace}}$ into $\resh{\cspace}\times\resh{\cspace}$ factor through some very familiar neighbourhoods of the diagonal: the formal neighbourhood in the case of $\fornei{\diago{\cspace}}$ and the germ of the diagonal in the case of $\infinei{\diago{\cspace}}$.

We start with looking at such neighbourhoods in general. Let $\Phi\colon\cspace'\rightarrow\cspace$ be a closed embedding,\footnote{Closed embeddings in $\cschemes$ correspond to surjective morphisms in $\allrings$.} extending (\ref{SetsOfIdeals}) we define
	\begin{equation*}\rfornil{\cspace}{\cspace'}:=\{\ideal\leq\cinfty(\cspace)\,|\,\nradical{\ideal}=\nradical{\kernel{\phi}}\},\end{equation*}
	\begin{equation*}\rinfinil{\cspace}{\cspace'}:=\{\ideal\leq\cinfty(\cspace)\,|\,\iradical{\ideal}=\iradical{\kernel{\phi}}\}.\end{equation*}
\begin{definition}\label{JetsAndGerms} Let $\phi\colon\cring\rightarrow\cring'$ be a surjective morphism in $\allrings$, and let $\Phi\colon\cspace'\rightarrow\cspace$ be the corresponding closed embedding in $\cschemes$. An element $a\in\cring$ has {\it vanishing (infinite) jet at $\cspace'$}, if $a\in\zerojetin{\cspace}{\cspace'}:=\underset{\ideal\in\rfornil{\cspace}{\cspace'}}\bigcap\ideal$. An element $a\in\cring$ has {\it vanishing germ at $\cspace'$}, if 
	\begin{equation*}a\in\zerogermin{\cspace}{\cspace'}:=
	\{a\in\cinfty(\cspace)\,|\,\exists b\in\cinfty(\cspace)\text{ s.t.\@ }a b=0\text{ and }\phi(b)\text{ is invertible}\}.\end{equation*}
{\it The (infinite) jet of $\cspace$ at $\cspace'$} is $\Jet{\cspace}{\cspace'}:=\spec(\cinfty(\cspace)/\zerojetin{\cspace}{\cspace'})$. {\it The germ of $\cspace$ at $\cspace'$} is $\Germ{\cspace}{\cspace'}:=\spec(\cinfty(\cspace)/\zerogermin{\cspace}{\cspace'})$.\footnote{To see that $\zerogermin{\cspace}{\cspace'}$ is an ideal let $a_1,a_2\in\cring$ together with $b_1,b_2\in\cring$, s.t.\@ $\phi(b_1),\phi(b_2)$ are invertible and $a_1 b_1=a_2 b_2=0$, it is clear that $(a_1 c_1+a_2 c_2)b_1 b_2=0$ for any $c_1,c_2\in\cring$, and $\phi(b_1 b_2)$ is invertible.}\end{definition}
If $\kernel{\phi}$ is a finitely generated ideal, it is clear that $\zerojetin{\cspace}{\cspace'}=\underset{k\in\mathbb N}\bigcap(\kernel{\phi})^k$. If $\cring$, $\cring'$ are finitely generated and germ-determined (\cite{MR91}, \S I.5), the usual notion of the germ of $\cspace$ at $\cspace'$ coincides with the one in Definition \ref{JetsAndGerms}. Indeed, choose a surjective $\alpha\colon\cinfty(\mathbb R^n)\rightarrow\cring$, and let $\alpha'\colon\cinfty(\mathbb R^n)\rightarrow\cring'$ be the composition. Let $\ideal,\ideal'\leq\cinfty(\mathbb R^n)$ be the kernels of $\alpha$, $\alpha'$ respectively, and let $V,V'\subseteq\mathbb R^n$ be the corresponding closed sets of common zeroes. 

Let $a\in\cring$ have $0$ germ at $\cspace'$ in the usual sense, this means $\exists U\subseteq\mathbb R^n$ open, s.t.\@ $U$ contains $V'$ and $a|_U=0$, i.e.\@ $\forall f\in\alpha^{-1}(a)$ we have $f|_U\in\ideal|_U$. We choose $g\in\cinfty(\mathbb R^n)$, s.t.\@ $g\neq0$ in an open $U'\supseteq V'$ and $g|_{\mathbb R^n\setminus W}=0$ for some closed $W\subseteq U$. Then $(f g)|_U\in\ideal|_U$, and this implies $f g\in\ideal$ ($\ideal$ is germ-determined). Conversely, suppose $f g\in\ideal$, and $\alpha'(g)$ is invertible. The latter implies that $g$ is invertible in some open $U\supseteq V'$. Dividing by $g|_U$ we find that $f|_U\in\ideal|_U$. 

\smallskip

Now we would like to connect this to elements of $\rinfinil{\cspace}{\cspace'}$.

\begin{proposition}\label{GermsThroughNeighbourhoods} Let $\phi\colon\cring\rightarrow\cring'$ be a surjective morphism in $\allrings$, and let $\Phi\colon\cspace'\rightarrow\cspace$ be the corresponding closed embedding in $\cschemes$. We have
	\begin{equation}\label{TwoVersionsOfGerms}\zerogermin{\cspace}{\cspace'}\subseteq\underset{\ideal\in\rinfinil{\cspace}{\cspace'}}\bigcap\ideal.\end{equation}
Suppose in addition that $\cring$ is finitely generated and point-determined (\cite{MR91}, \S I.5). Then (\ref{TwoVersionsOfGerms}) is an equality.\end{proposition}
\begin{proof} Let $\ideal\in\rinfinil{\cspace}{\cspace'}$ and let $a,b\in\cring$, s.t.\@ $\phi(b)$ is invertible and $a b=0$. Clearly $\phi(b)$ is invertible $\Leftrightarrow$ $\igen{b}+\kernel{\phi}=\cring$. Since $\iradical{-}$ preserves inclusions of ideals, $\cring=\igen{b}+\iradical{\ideal}\leq\iradical{\igen{b}}+\iradical{\ideal}\leq\iradical{\igen{b}+\ideal}$, i.e.\@ $\iradical{\igen{b}+\ideal}=\cring$. This implies $\igen{b}+\ideal=\cring$. Indeed, choosing generators for $\cring$ we have a surjective $\alpha\colon\cinfty(\mathbb R^S)\rightarrow\cring$, and using the fact that $\iradical{\alpha^{-1}(\igen{b}+\ideal)}=\alpha^{-1}(\iradical{\igen{b}+\ideal})$ (Lemma \ref{SurjectiveFunctoriality}) we conclude that $\iradical{\alpha^{-1}(\igen{b}+\ideal)}=\cinfty(\mathbb R^S)$. From Lemma 2.2 in \cite{MR86} it follows that $\exists c\in\alpha^{-1}(\igen{b}+\ideal)$, s.t.\@ $c$ vanishes nowhere in $\cinfty(\mathbb R^S)$, i.e.\@ $c$ is invertible, meaning $\alpha^{-1}(\igen{b}+\ideal)=\cinfty(\mathbb R^S)$. Since $b$ becomes invertible in $\cring/\ideal$, and $a b=0$, it must be that $a\in\ideal$.

To prove the second part of the proposition we need to look at rates of vanishing at $\cspace'$ of upper bounds of functions. The idea is simple: if the upper bound of a function vanishes faster than any infinite order of vanishing, the function must be $0$ in a neighbourhood of the closed set. The details are rather technical and we collect them in Appendix \ref{ProofGermsThroughNeighbourhoods}.\end{proof}

From Lemma \ref{Facts} we know that the sets of ideals $\rfornil{\cspace}{\cspace'}$ and $\rinfinil{\cspace}{\cspace'}$ are closed with respect to finite intersections, hence they are fundamental systems of neighbourhoods for some linear topologies on $\cring$.  Proposition \ref{GermsThroughNeighbourhoods} tells us that $\zerogermin{\cspace}{\cspace'}\leq\underset{\ideal\in\rinfinil{\cspace}{\cspace'}}\bigcap\ideal$, i.e.\@ $\zerogermin{\cspace}{\cspace'}$ is in the closure of $\{0\}\subseteq\cring$. Altogether we have the following definition.

\begin{definition} Let $\cring\rightarrow\cring'$ be a surjective morphism in $\allrings$, and let $\cspace'\subseteq\cspace$ be the corresponding closed embedding in $\cschemes$. We write $(\cring/\zerojetin{\cspace}{\cspace'})^\tau$, $(\cring/\zerogermin{\cspace}{\cspace'})^\tau$ to mean the corresponding $\cinfty$-rings, equipped with the linear topologies given by $\rfornil{\cspace}{\cspace'}$ and  $\rinfinil{\cspace}{\cspace}$ respectively. 
\end{definition}
It is clear that the topology on $(\cring/\zerojetin{\cspace}{\cspace'})^\tau$ is separated. From Proposition \ref{GermsThroughNeighbourhoods} it follows that also $(\cring/\zerogermin{\cspace}{\cspace'})^\tau$ has separated topology, provided $\cring$ is finitely generated and point-determined.

\begin{definition} {\it The pro-nilpotent neighbourhood of $\cspace'$ in $\cspace$} is the pre-sheaf on $\cschemes$ defined as follows: $\forall\cspace''\in\cschemes$
	 \begin{equation}\label{FormalNeighbourhood}\fornein{\cspace}{\cspace'}(\cspace''):=
	 \hom_{\rm cont}((\cring/\zerojetin{\cspace}{\cspace'})^\tau,\cinfty(\cspace'')),\end{equation}
where $\hom_{\rm cont}$ stands for continuous $\cinfty$-morphisms, with $\cinfty(\cspace'')$ equipped with the discrete topology. Similarly, {\it the $\infty$-nilpotent neighbourhood of $\cspace'$ in $\cspace$} is the pre-sheaf
	\begin{equation}\label{GermNeighbourhood}\infinein{\cspace}{\cspace'}(\cspace''):=
	\hom_{\rm cont}((\cring/\zerogermin{\cspace}{\cspace'})^\tau,\cinfty(\cspace'')).\end{equation}
\end{definition}
There are of course the natural morphisms
	\begin{equation*}\fornein{\cspace}{\cspace'}\longrightarrow\resh{\spec(\cring/\zerojetin{\cspace}{\cspace'})},\quad
	\infinein{\cspace}{\cspace'}\longrightarrow\resh{\spec(\cring/\zerogermin{\cspace}{\cspace'})},\end{equation*}
which are not isomorphisms in general. In fact $\resh{\spec(\cring/\zerojetin{\cspace}{\cspace'})}$ is a reflection of $\fornein{\cspace}{\cspace'}$ in the subcategory of representable pre-sheaves (\cite{Bo94}, Def.\@ 3.1.1). Similarly $\resh{\spec(\cring/\zerogermin{\cspace}{\cspace'})}$ is a reflection of $\infinein{\cspace}{\cspace'}$ in the subcategory of representable pre-sheaves over $\cspace$,\footnote{Notice that, different from the formal case, we get a reflection in the category over $\cspace$. This is because every power series can be summed to a $\cinfty$-function (Borel lemma), but in the $\infty$-case we might have obstructions to integrability. Questions of integrability become important in \cite{BP2}.}  provided $\cring$ is finitely generated and point-determined (Proposition \ref{GermsThroughNeighbourhoods}). All this is standard, and we provide the following lemma only for completeness.

\begin{lemma} Let $\mathcal C$ be a category having all colimits, and let $\mathcal F\colon\mathcal D\rightarrow\mathcal C$ be a diagram. Let $\yoneda\colon\mathcal C\rightarrow\hom(\mathcal C^{\op},\sets)$ be the covariant Yoneda embedding. Then $\underset{\mathcal D}\colim\;\mathcal F$ is a reflection of $\underset{\mathcal D}\colim\;\yoneda\circ\mathcal F$ in the subcategory of representable pre-sheaves.\end{lemma}
\begin{proof} Let $C\in\mathcal C$, morphisms $\underset{\mathcal D}\colim\;\yoneda\circ\mathcal F\rightarrow\yoneda(C)$ correspond to co-cones $\{\yoneda(\mathcal F(D))\rightarrow\yoneda(C)\}_{D\in\mathcal D}=\{\mathcal F(D)\rightarrow C\}_{D\in\mathcal D}$. Therefore there is a morphism $\underset{\mathcal D}\colim\;\yoneda\circ\mathcal F\rightarrow\yoneda(\underset{\mathcal D}\colim\;\mathcal F)$, which is unique up to a unique isomorphism of $\underset{\mathcal D}\colim\;\mathcal F$. For the same reason every representable co-cone on $\yoneda\circ\mathcal F$ uniquely factors through $\yoneda(\underset{\mathcal D}\colim\;\mathcal F)$.\end{proof}%

\smallskip

We can apply all this to the diagonal embedding of some $\cspace\in\cschemes$.

\begin{proposition} If $\cspace\in\cschemes$ is $\redun$-injective, the canonical $\resh{\cspace}\rightarrow\algered{\resh{\cspace}}$ is a co-equalizer of $\fornein{\cspace\times\cspace}{\diago{\cspace}}\rightrightarrows\resh{\cspace}$. The reflection of $\fornein{\cspace\times\cspace}{\diago{\cspace}}$ in $\cschemes$ is the formal neighbourhood of $\diago{\cspace}$ in $\cspace\times\cspace$.

If $\cspace\in\cschemes$ is $\redui$-injective, the canonical $\resh{\cspace}\rightarrow\infired{\resh{\cspace}}$ is a co-equalizer of $\infinein{\cspace\times\cspace}{\diago{\cspace}}\rightrightarrows\resh{\cspace}$. If $\cspace\times\cspace$ is finitely generated and point-determined, the reflection of $\infinein{\cspace\times\cspace}{\diago{\cspace}}$ in $\cschemes/(\cspace\times\cspace)$ is the germ of $\cspace\times\cspace$ at $\diago{\cspace}$.\end{proposition}
\begin{proof} If $\cspace$ is $\redun$-injective, from Proposition \ref{DeRhamAsQuotient} we know that $\resh{\cspace}\rightarrow\algered{\resh{\cspace}}$ is a co-equalizer of $\fornei{\diago{\cspace}}\rightrightarrows\resh{\cspace}$, where 
	\begin{equation*}\fornei{\diago{\cspace}}(\cspace')=
	\underset{\ideal\in\rinfinil{\cspace\times\cspace}{\cspace}}\colim\hom_{\allrings}(\cinfty(\cspace\times\cspace)/\ideal,\cinfty(\cspace')),\end{equation*}
which coincides with (\ref{FormalNeighbourhood}) with the appropriate change of notation. Similarly, assuming $\cspace$ is $\redui$-injective, we use Proposition \ref{DeRhamAsQuotient} to arrive at (\ref{GermNeighbourhood}).\end{proof}%

\begin{example} Let $\cmanif$ be a manifold embeddable into some $\mathbb R^n$. Then it is finitely presented (\cite{MR91} Thm.\@ I.2.3) and point determined (\cite{MR91} Thm.\@ I.4.2), moreover $\cmanif\times\cmanif$, computed in $\cschemes$, is also a manifold (\cite{MR91}, Prop.\@ I.2.6). We have already seen (Example \ref{ManifoldsInjective}) that, being a retract of an open subset in $\mathbb R^n$, $\cmanif$ is $\redui$-injective. Therefore $\resh{\cmanif}\rightarrow\infired{\cmanif}$ is a coequalizer of $\infinein{\cmanif\times\cmanif}{\diago{\cmanif}}\rightrightarrows\resh{\cmanif}$, with the reflection of $\infinein{\cmanif\times\cmanif}{\diago{\cmanif}}$ in representable pre-sheaves being the germ of $\cmanif\times\cmanif$ at $\diago{\cmanif}$.\end{example}

\subsection{De Rham groupoids}\label{DeRhamGroupoids}

So far we have considered only pre-sheaves on $\cschemes$. Now we would like to have a Grothendieck topology and work with sheaves.

\begin{definition} (E.g.\@ \cite{MR91} \S  VI.1) Let $\fgrings\subset\allrings$ be the full subcategory of finitely generated $\cinfty$-rings. {\it The Zariski topology} on $\fschemes$ is defined as follows: $\forall\cspace\in\fschemes$ a set of morphisms $\{\Phi_i\colon\cspace_i\rightarrow\cspace\}_{i=1}^k$ $k\in\mathbb Z_{\geq 0}$ is {\it a covering family}, if $\forall i\exists a_i\in\cinfty(\cspace)$, s.t.\@ $\phi_i\colon\cinfty(\cspace)\rightarrow\cinfty(\cspace_i)$ is the universal morphism inverting $a_i$, and $\underset{1\leq i\leq k}\sum\;\igen{a_i}=\cinfty(\cspace)$.
\end{definition}
As usual with Zariski topologies, this topology is sub-canonical (e.g.\@ \cite{MR91} Lemma VI.1.3). We denote all $\sets$-valued functors on $\fschemes$ by $\Pre$, and the full subcategory of sheaves by $\Shea$.

\begin{example} Let $\cspace'\subseteq\cspace$ be a closed embedding in $\fschemes$. We have constructed two neighbourhoods of $\resh{\cspace'}$ in $\resh{\cspace}$: the pro-nilpotent $\fornein{\cspace}{\cspace'}$ and the $\infty$-nilpotent $\infinein{\cspace}{\cspace'}$. From Proposition \ref{FilteredDiagrams} we know that these pre-sheaves are colimits of filtered diagrams of representable pre-sheaves. Since each covering family in Zariski topology is finite and filtered colimits commute with finite limits, we conclude that both $\fornein{\cspace}{\cspace'}$ and $\infinein{\cspace}{\cspace'}$ are in fact sheaves. In general the sheafification functor commutes with arbitrary colimits, hence these neighbourhoods are also colimits computed in the category of sheaves. Moreover, they are sub-sheaves of $\resh{\cspace}$, i.e.\@ $\forall\cspace''\in\fschemes$ the maps of sets
	\begin{equation*}\xymatrix{\fornein{\cspace}{\cspace'}(\cspace'')\ar[r] & (\resh{\cspace})(\cspace'') & 
	\infinein{\cspace}{\cspace'}(\cspace'')\ar[l]}\end{equation*}
are injective. Indeed, let $\Phi\in\infinein{\cspace}{\cspace'}(\cspace'')$. By definition 
	\begin{equation*}\infinein{\cspace}{\cspace'}(\cspace''):=
	\underset{\ideal\in\rinfinil{\cspace}{\cspace'}}\colim\hom_{\allrings}(\cinfty(\cspace)/\ideal,\cinfty(\cspace'')),\end{equation*}
hence $\Phi$ is given by $\phi\colon\cinfty(\cspace)/\ideal\rightarrow\cinfty(\cspace'')$ for some $\ideal\in\rinfinil{\cspace}{\cspace'}$ (this is just to say that a singleton is a compact object in $\sets$). Having two
	\begin{equation*}\phi_1\colon\cinfty(\cspace)/\ideal_1\longrightarrow\cinfty(\cspace''),\quad
	\phi_2\colon\cinfty(\cspace)/\ideal_2\longrightarrow\cinfty(\cspace''),\end{equation*}
s.t.\@ $\phi_1$, $\phi_2$ are mapped to the same element of $\hom_{\allrings}(\cinfty(\cspace),\cinfty(\cspace'')$, we see that their images in $\hom_{\allrings}(\cinfty(\cspace)/\ideal_1\cap\ideal_2,\cinfty(\cspace''))$ are also equal, i.e.\@ $\phi_1=\phi_2$ as elements of $\infinein{\cspace}{\cspace'}$. The case of $\fornein{\cspace}{\cspace'}$ is similar. The following lemma shows functoriality of these sub-sheaves.\end{example}

\begin{lemma}\label{PairFunctoriality} Let $\embeddings{\cschemes}$ be the category of closed embeddings $\Phi\colon\cspace'\hookrightarrow\cspace$ in $\cschemes$ with morphisms $\Phi_1\rightarrow\Phi_2$ being commutative squares 
	\begin{equation}\label{MorPair}\xymatrix{\cspace'_1\ar[d]_{\Phi_1}\ar[rr] && \cspace'_2\ar[d]^{\Phi_2}\\ \cspace_1\ar[rr] && \cspace_2.}\end{equation} 
Let $\shpairs$ be the category of inclusions of sub-presheaves on $\cschemes$. The constructions $\fornein{\cspace}{\cspace'}\hookrightarrow\resh{\cspace}$, $\infinein{\cspace}{\cspace'}\hookrightarrow\resh{\cspace}$ extend to functors 
	\begin{equation*}\embeddings{\cschemes}\rightarrow\shpairs,\quad\embeddings{\cschemes}\rightarrow\shpairs.\end{equation*}
\end{lemma}
\begin{proof} Consider (\ref{MorPair}), and apply the covariant Yoneda embedding. We claim that the composite $\infinein{\cspace_1}{\cspace'_1}\rightarrow\resh{\cspace_2}$ factors through $\infinein{\cspace_2}{\cspace'_2}$. It is enough to show that $\forall\ideal_1\in\rinfinil{\cspace_1}{\cspace'_1}$ there is $\ideal_2\in\rinfinil{\cspace_2}{\cspace'_2}$, s.t.\@ $\cinfty(\cspace_2)\rightarrow\cinfty(\cspace_1)/\ideal_1$ has $\ideal_2$ in the kernel. Let $\ideal_1'\leq\cinfty(\cspace_2)$ be the kernel of $\cinfty(\cspace_2)\rightarrow\cinfty(\cspace_1)/\ideal_1$. Since $\iradical{\ideal_1}=\iradical{\kernel{\phi_1}}$, from Proposition \ref{StrongFunctoriality} we have that $\iradical{\ideal'_1}$ is the kernel of $\cinfty(\cspace_2)\rightarrow\redui(\cinfty(\cspace'_1))$. Then from (\ref{MorPair}) it is clear that $\iradical{\ideal'_1}\geq\kernel{\phi_2}$. From Lemma \ref{Facts} we have $\iradical{\ideal'_1\cap\kernel{\phi_2}}=\iradical{\ideal'_1}\cap\iradical{\kernel{\phi_2}}=\iradical{\kernel{\phi_2}}$, i.e.\@ $\ideal'_1\cap\kernel{\phi_2}\in\rinfinil{\cspace_2}{\cspace'_2}$. Putting $\ideal_2:=\ideal'_1\cap\kernel{\phi_2}$ we have shown that $\infinein{\cspace_1}{\cspace'_1}\rightarrow\resh{\cspace_2}$ factors through $\infinein{\cspace_2}{\cspace'_2}$. Functoriality now follows from the fact that $\infinein{\cspace_2}{\cspace'_2}$ is a sub-presheaf of $\resh{\cspace_2}$. The case of $\fornein{\cspace}{\cspace'}$ is analogous.\end{proof}

We recall a very familiar groupoid.

\begin{definition} {\it The pair groupoid on $\cspace\in\cschemes$} is $(\cspace,\cspace\times\cspace)$ with the source, target, identity morphisms being the left projection, right projection and the diagonal respectively. The composition morphism is given by the projection $\cspace^{\times^3}\rightarrow\cspace^{\times^2}$ on the first and last factors.

{\it The nerve of this groupoid} will be denoted by $\cspace^\bullet:=\{\cspace^{\times^k}\}_{k\in\mathbb N}$, with the simplicial structure morphisms given by all projections and diagonals.\end{definition}
Applying the covariant Yoneda embedding we get a groupoid object $(\resh{\cspace},\resh{\cspace}\times\resh{\cspace})$. We denote by $\fornein{\cspace^{\times^m}}{\cspace}$, $\infinein{\cspace^{\times^m}}{\cspace}$ the pro-nilpotent and $\infty$-nilpotent neighbourhoods of the main diagonal in $\resh{\cspace}^{\times^m}$.

\begin{lemma}\label{ProjectionNeigh} Let $\cspace\in\cschemes$, and let $m>n$ in $\mathbb Z_{\geq 2}$. Let $\projection\colon\cspace^{\times^m}\rightarrow\cspace^{\times^n}$ be projection on a factor. The images\footnote{Here we use the sheaf-theoretic notion of an image for (pre-)sheaves on $\cschemes$.} of $\fornein{\cspace^{\times^m}}{\cspace}\rightarrow\resh{\cspace}^{\times^n}$,  $\infinein{\cspace^{\times^m}}{\cspace}\rightarrow\resh{\cspace}^{\times^n}$ are $\fornein{\cspace^{\times^n}}{\cspace}$ and $\infinein{\cspace^{\times^n}}{\cspace}$ respectively.\end{lemma}
\begin{proof} Up to a permutation of the factors, we can assume that $\pi$ is the projection on the first $n$ factors in $\cspace^{\times^m}$. Mapping the $n$-th factor diagonally into $\cspace^{\times^{m-n+1}}$ we obtain a section $\injection\colon\cspace^{\times^n}\rightarrow\cspace^{\times^m}$ of $\projection$. It is easy to see that $\injection$ maps the main diagonal in $\cspace^{\times^n}$ to the main diagonal in $\cspace^{\times^m}$. In other words the diagonal embedding $\cspace\rightarrow\cspace^{\times^n}$ is a retract of the diagonal embedding $\cspace\rightarrow\cspace^{\times^m}$. Applying the infinitesimal neighbourhoods constructions we immediately conclude (Lemma \ref{PairFunctoriality}) that $\fornein{\cspace^{\times^n}}{\cspace}$ is the image of $\fornein{\cspace^{\times^m}}{\cspace}$, and correspondingly $\infinein{\cspace^{\times^n}}{\cspace}$ is the image of $\infinein{\cspace^{\times^m}}{\cspace}$.
\end{proof}%

Now we consider the projections $\fornein{\cspace^{\times^m}}{\cspace}\rightrightarrows\resh{\cspace}$, $\infinein{\cspace^{\times^m}}{\cspace}\rightrightarrows\resh{\cspace}$ on the first and last factors. The following lemma is almost obvious.

\begin{lemma}\label{ProductNeigh} Let $\cspace\in\cschemes$, $m,n\geq 2$. In $\Pre/\resh{\cspace}^{\times^{m+n-1}}$ the objects
	\begin{equation*}\forneired{\cspace^{\times^m}}{\cspace}\underset{\resh{\cspace}}\times\forneired{\cspace^{\times^n}}{\cspace}
	\longrightarrow \resh{\cspace}^{\times^{m+n-1}}, \quad
	\infineired{\cspace^{\times^m}}{\cspace}\underset{\resh{\cspace}}\times\infineired{\cspace^{\times^n}}{\cspace}
	\longrightarrow \resh{\cspace}^{\times^{m+n-1}}\end{equation*}
are isomorphic to $\forneired{\cspace^{\times^{m+n-1}}}{\cspace}$ and $\infineired{\cspace^{\times^{m+n-1}}}{\cspace}$ respectively.\end{lemma}
\begin{proof} From functoriality of pullbacks we get $\infinein{\cspace^{\times^m}}{\cspace}\underset{\resh{\cspace}}\times\infinein{\cspace^{\times^n}}{\cspace}\rightarrow\resh{\cspace}^{\times^{m+n-1}}$. We claim this morphism factors through $\infinein{\cspace^{\times^{m+n-1}}}{\cspace}$. Both $\infinein{\cspace^{\times^{m}}}{\cspace}$ and $\infinein{\cspace^{\times^{n}}}{\cspace}$ are colimits in $\Pre$, the former is over a diagram parameterized by $\rinfinil{\cspace^{\times^m}}{\cspace}$, the latter is over a diagram parameterized by $\rinfinil{\cspace^{\times^n}}{\cspace}$. Consider the category $\rinfinil{\cspace^{\times^m}}{\cspace}\times\rinfinil{\cspace^{\times^n}}{\cspace}$. Composing the two projections to the factors with the functors to the category of pre-sheaves on $\fschemes$ we realize $\infinein{\cspace^{\times^{m}}}{\cspace}$ and $\infinein{\cspace^{\times^{n}}}{\cspace}$ as colimits over diagrams parameterized by the same category. Since product of filtered categories is again filtered, and colimits over filtered categories commute with finite limits (in $\Pre$), we can represent $\infinein{\cspace^{\times^m}}{\cspace}\underset{\resh{\cspace}}\times\infinein{\cspace^{\times^n}}{\cspace}$ as a colimit of
	\begin{equation*}\{\resh{\spec(\cinfty(\cspace^{\times^m})/\ideal')}\underset{\resh{\cspace}}\times
	\resh{\spec(\cinfty(\cspace^{\times^n})/\ideal'')}\},\end{equation*}
where $\ideal'$, $\ideal''$ run over $\rinfinil{\cspace^{\times^m}}{\cspace}$ and $\rinfinil{\cspace^{\times^n}}{\cspace}$ respectively. Since Yoneda embedding preserves limits, we can compute these pullbacks in $\cschemes$, and we obtain $\{\spec(\cinfty(\cspace^{\times^{m+n-1}})/\ideal'+\ideal'')\}$. Let $\ideal_{\Delta^m}$ be the ideal of the diagonal in $\cspace^{\times^m}$. By assumption $\iradical{\ideal'}=\iradical{\ideal_{\Delta^m}}$, and similarly $\iradical{\ideal''}=\iradical{\ideal_{\Delta^n}}$. We note that $\ideal_{\Delta^m}+\ideal_{\Delta^n}=\ideal_{\Delta^{m+n-1}}$ in $\cinfty(\cspace^{\times^{m+n-1}})$, indeed the ideal of each diagonal is generated by $\{a_i-a_j\}$, where the subscripts correspond to different copies of $\cinfty(\cspace)$ and $a\in\cinfty(\cspace)$. Taking coproduct over $\cinfty(\cspace)$, considered as the $m$-th copy, generates $\ideal_{\Delta^{m+n-1}}$ from $\ideal_{\Delta^m}$ and $\ideal_{\Delta^n}$. Therefore, since according to Lemma \ref{Facts}
	\begin{equation*}\iradical{\ideal'+\ideal''}\geq\iradical{\ideal'}+\iradical{\ideal''}=\iradical{\ideal_{\Delta^m}}+\iradical{\ideal_{\Delta^n}}\geq
	\ideal_{\Delta^m}+\ideal_{\Delta^n},\end{equation*}
we have $\ideal'+\ideal''\in\rinfinil{\cspace^{\times^{m+n-1}}}{\cspace}$, and hence $\infinein{\cspace^{\times^m}}{\cspace}\underset{\resh{\cspace}}\times\infinein{\cspace^{\times^n}}{\cspace}\rightarrow\resh{\cspace}^{\times^{m+n-1}}$ factors through $\infinein{\cspace^{\times^{m+n-1}}}{\cspace}\rightarrow\resh{\cspace}^{\times^{m+n-1}}$. We claim that
	\begin{equation}\label{ProductNeighbourhoods}
	\xymatrix{\infinein{\cspace^{\times^m}}{\cspace}\underset{\resh{\cspace}}\times\infinein{\cspace^{\times^n}}{\cspace}
	\ar[r] & \infinein{\cspace^{\times^{m+n-1}}}{\cspace}}\end{equation}
is an isomorphism. From Lemma \ref{ProjectionNeigh} it follows that $\forall\,\ideal\in\rinfinil{\cspace^{\times^{m+n-1}}}{\cspace}$ there are $\ideal'\in\rinfinil{\cspace^{\times^m}}{\cspace}$, $\ideal''\in\rinfinil{\cspace^{\times^n}}{\cspace}$, s.t.\@ $\ideal'+\ideal''\leq\ideal$. As $\ideal'+\ideal''\in\rinfinil{\cspace^{\times^{m+n-1}}}{\cspace}$ for any $\ideal'\in\rinfinil{\cspace^{\times^m}}{\cspace}$ and $\ideal''\in\rinfinil{\cspace^{\times^n}}{\cspace}$, the domain of (\ref{ProductNeighbourhoods}) is a colimit over a right cofinal sub-diagram of the diagram used to compute the codomain of (\ref{ProductNeighbourhoods}) (\cite{Hirschhorn03}, Def.\@ 14.2.1) and hence the colimits agree (\cite{Hirschhorn03} Thm.\@ 14.2.5).\end{proof}

Notice that the topological $\cinfty$-ring corresponding to $\infineired{\cspace}{\cspace}$ has discrete topology ($0\in\rinfinil{\cspace}{\cspace}$), and hence $\infineired{\cspace}{\cspace}\cong\cspace$. Similarly $\forneired{\cspace}{\cspace}\cong\cspace$. We denote 
	\begin{equation*}\infineired{\cspace^\bullet}{\cspace}:=\{\infineired{\cspace^{\times^{k+1}}}{\cspace}\}_{k\in\mathbb Z_{\geq 0}},\quad
	\forneired{\cspace^\bullet}{\cspace}:=\{\forneired{\cspace^{\times^{k+1}}}{\cspace}\}_{k\in\mathbb Z_{\geq 0}}.\end{equation*}
From the previous lemmas together we have the following proposition.

\begin{proposition}\label{Nerves} Let $\cspace\in\cschemes$, the morphisms $\forneired{\cspace\times\cspace}{\cspace}\rightarrow\resh{\cspace}\times\resh{\cspace}$, $\infineired{\cspace\times\cspace}{\cspace}\rightarrow\resh{\cspace}\times\resh{\cspace}$ define sub-groupoids of $(\resh{\cspace},\resh{\cspace}\times\resh{\cspace})$. Their nerves are $\forneired{\cspace^\bullet}{\cspace}$ and $\infineired{\cspace^\bullet}{\cspace}$ respectively.\end{proposition}
\begin{proof} Both $\fornein{\cspace\times\cspace}{\cspace}$ and $\infinein{\cspace\times\cspace}{\cspace}$ are colimits of neighbourhoods of the diagonal $\diago{\resh{\cspace}}$ in $\resh{\cspace}\times\resh{\cspace}$, therefore the diagonal morphism $\resh{\cspace}\rightarrow\resh{\cspace}\times\resh{\cspace}$ factors through each one of them. Composing the embeddings $\fornein{\cspace\times\cspace}{\cspace}\rightarrow\resh{\cspace}\times\resh{\cspace}$, $\infinein{\cspace\times\cspace}{\cspace}\rightarrow\resh{\cspace}\times\resh{\cspace}$ with the projections on $\resh{\cspace}$ we get the source and target morphisms, which are left inverses of the diagonal. From Lemma \ref{ProductNeigh} it follows that the composition morphism on $(\resh{\cspace},\resh{\cspace}\times\resh{\cspace})$ descends to $\fornein{\cspace\times\cspace}{\cspace}$ and $\infinein{\cspace\times\cspace}{\cspace}$. The same lemma gives us also the description of the nerves.\end{proof}

\smallskip

Let $\cspace\in\cschemes$. Recall that we denote $\resh{\cspace}^\bullet:=\{\resh{\cspace}^{\times^{k+1}}\}_{k\in\mathbb Z_{\geq 0}}$, together with all the projection and diagonal maps. This is a simplicial object in $\Pre$ (in fact in $\Shea$), and we can compute the corresponding sheaves of homotopy groups. For any $\cspace'\in\cschemes$ the simplicial set $\resh{\cspace}^\bullet(\cspace')$ is the nerve of the contractible groupoid on the set $\hom_{\cschemes}(\cspace',\cspace)$. Therefore this is a contractible Kan complex. So $\resh{\cspace}^\bullet$ is not much more than $\resh{\spec(\mathbb R)}$. To be precise we need a model structure.

\begin{definition} (\cite{Hirschhorn03}, \cite{Blander03}, \cite{To10}) {\it The global model structure on $\Pre$} is defined by objectwise weak equivalences and objectwise fibrations. {\it The local model structure on $\Pre$} has global cofibrations as cofibrations, and its weak equivalences are those morphisms that induce isomorphisms of {\it sheaves} of homotopy groups.\end{definition}
Thus to see, if a morphism is a local weak equivalence, we need first to compute the pre-sheaves of homotopy groups, and then sheafify. In particular, every global weak equivalence is also a local weak equivalence. For example $\forall\cspace\in\cschemes$ the unique morphism $\resh{\cspace}^\bullet\rightarrow\resh{\rpt}$ is a global weak equivalence. Therefore it is also a local weak equivalence. We proceed similarly in the following proposition.

\begin{proposition} Suppose $\cspace\in\cschemes$ is $\redun$-injective (respectively $\redui$-injective). Consider $\algered{\resh{\cspace}}$ (respectively $\infired{\resh{\cspace}}$) as a constant simplicial object in $\Pre$. Then $\forneired{\cspace^\bullet}{\cspace}\rightarrow\algered{\resh{\cspace}}$ (respectively $\infineired{\cspace^\bullet}{\cspace}\rightarrow\infired{\resh{\cspace}}$) is both a global and a local weak equivalence in $(\Pre)^{\simplop}$.\end{proposition}
\begin{proof} From Proposition \ref{Nerves} we know that $\infineired{\cspace^\bullet}{\cspace}$ is the nerve of $(\cspace, \infineired{\cspace^{\times^2}}{\cspace})$, therefore $\forall\cspace'\in\cschemes$ the simplicial set $\infinein{\cspace^\bullet}{\cspace}(\cspace')$ is a Kan complex, and it has vanishing homotopy groups in dimension $\geq 2$. Since $\resh{\cspace}^\bullet$ is a contractible groupoid, it is easy to show that also $\pi_1(\infineired{\cspace^\bullet}{\cspace}(\cspace'))$ is trivial. Indeed, let $\cspace'\in\cschemes$, and let $\Phi\colon\cspace'\rightarrow\infinein{\cspace^{\times^2}}{\cspace}$, s.t.\@ the two compositions $\projection_1\circ\Phi,\projection_2\circ\Phi\colon\cspace'\rightarrow\cspace$ are equal. From the definition of $\infinein{\cspace^{\times^2}}{\cspace}$ it follows $\exists\ideal\in\rinfinil{\cspace^{\times^2}}{\cspace}$, s.t.\@ $\Phi$ can be represented by $\phi\colon\cinfty(\cspace\times\cspace)/\ideal\rightarrow\cinfty(\cspace')$, and $\projection_1\circ\Phi=\projection_2\circ\Phi$ means that $\phi\circ\injection_1=\phi\circ\injection_2$, where $\injection_1,\injection_2\colon\cinfty(\cspace)\rightarrow\cinfty(\cspace\times\cspace)/\ideal$ are the two inclusions, followed by dividing by $\ideal$. Thus the composite morphism $\cinfty(\cspace\times\cspace)\longrightarrow\cinfty(\cspace\times\cspace)/\ideal\overset{\phi}{\longrightarrow}\cinfty(\cspace')$ factors through the diagonal $\cinfty(\cspace\times\cspace)\rightarrow\cinfty(\cspace)$. Since the ideal of the diagonal contains $\ideal$, this factorization defines a factorization of $\phi$, in other words $\Phi$ is a degenerate simplex. Finally Proposition \ref{DeRhamAsQuotient} tells us that $\pi_0(\infinein{\cspace^\bullet}{\cspace}(\cspace'))=\infired{\resh{\cspace}}(\cspace')$. The nilpotent case is similar.\end{proof}%

\section{Some differential operators and further questions}\label{DifferentialOperators}

Let $\cspace\in\cschemes$, and suppose it is $\redun$- and $\redui$-injective. For example any manifold would do. We can construct 3 groupoids:\begin{itemize}
\item[1.] the pair groupoid $(\cspace,\cspace\times\cspace)$,
\item[2.] the formal neighbourhood of the diagonal $(\resh{\cspace},\forneired{\cspace\times\cspace}{\cspace})$,
\item[3.] the germ of the diagonal $(\resh{\cspace},\infineired{\cspace\times\cspace}{\cspace})$.\end{itemize} 
The case of the pair groupoid is very familiar. Here any $\mathbb R$-point of $\cspace$ can be mapped to any other. Sections of $\projection_1\colon\cspace\times\cspace\rightarrow\cspace$ over all of $\cspace$ correspond to smooth morphisms $\cspace\rightarrow\cspace$. It is clear that the full pair groupoid cannot give a meaningful infinitesimal theory.

\smallskip

The formal neighbourhood of the diagonal is also well known. There is only one section of $\projection_1\colon\forneired{\cspace\times\cspace}{\cspace}\rightarrow\cspace$. It is the diagonal. Thus here we do get an infinitesimal theory, which is the usual theory of polynomial differential operators of course. Indeed, for $\cspace':=\cspace\times\spec(\cinfty(\mathbb R)/(\zerovalue{0})^n)$ diagrams 
	\begin{equation*}\xymatrix{\cspace'\ar[rr]\ar[rd]_{\projection_1} && \fornein{\cspace\times\cspace}{\cspace}\ar[ld]^{\projection_1}\\ 
	& \cspace &}\end{equation*}
are in one-to-one correspondence with differential operators of order $<n$.

\smallskip

The germ of the diagonal is new. Also here the diagonal is the only section of $\projection_1\colon\infineired{\cspace\times\cspace}{\cspace}\rightarrow\cspace$, but we get considerably more parameterized sections than in the formal case. Let $\Phi\colon\mathbb R^k\times\mathbb R\rightarrow\mathbb R^k$ be a $1$-parameter family of smooth maps $\mathbb R^k\rightarrow\mathbb R^k$, s.t.\@ $\Phi|_{\mathbb R^k\times\{0\}}$ is $\id_{\mathbb R^k}$. It is not difficult to construct an example such that restriction of $\Phi$ to the formal neighbourhood of $\mathbb R^k\times\{0\}$ in $\mathbb R^k\times\mathbb R$ factors through the projection on $\mathbb R^k$, yet for no point $\rpt\in\mathbb R\setminus\{0\}$ is $\Phi|_{\mathbb R^k\times\{\rpt\}}$ the identity on $\mathbb R^k$.

Such $\Phi$ gives a section of $\forneired{\mathbb R^{2k}}{\mathbb R^k}$ that factors through the diagonal, yet the corresponding section of $\infineired{\mathbb R^{2k}}{\mathbb R^k}$ does not. In other words $\Phi$ is inaccessible by perturbation, yet it is visible through the $\infty$- de Rham groupoid.

\smallskip

As we can see $\infty$- de Rham groupoids give infinitesimal theory beyond perturbation, and we would like to develop the corresponding theory of \D-modules. Such development cannot fit within this paper. 

Of course, the obvious way is to define \D-modules as quasi-coherent sheaves on the de Rham groupoid. The problem is that the notion of a quasi-coherent sheaf of modules is not at all straightforward in $\cinfty$-geometry. The straightforward definition works well only for coherent sheaves. This is enough to define polynomial differential operators, since all we require are vector bundles of finite rank on the diagonal. The germ of the diagonal forces us to go beyond coherent sheaves. This is the subject of \cite{QuasiCoherent}.

\appendix

\section{Proof or Lemma \ref{RelativeOrders}}\label{ProofRelativeOrders}

We need to find $h\in\cinfty(\mathbb R^m)$, s.t.\@ $\vset{h}{\rpt}=V$ and $\psi(h)$ is divisible by $f$ within $\cinfty(\mathbb R^{m+n})$. Were $\mathbb R^{m+n}$ compact, we could do this by estimating the rate of vanishing of $f$ at $\Psi^{-1}(V)$. As $\mathbb R^{m+n}$ is not compact, we cut it into overlapping compact pieces using closed balls: $\forall s\in\mathbb R$ let $B_s\subseteq\mathbb R^m$, $B'_s\subseteq\mathbb R^n$ be closed balls of radii $s$ centered at the origins, if $s<0$ $B_{s}=B'_{s}:=\emptyset$.

To avoid dealing with signs we estimate from below the rate of vanishing of $f^2$. Let $\chi\in\cinfty(\mathbb R^m)$ be non-negative and s.t.\@ $\vset{\chi}{\rpt}=V$, we define (possibly discontinuous) functions $\{\lambda_{s,t}\colon\mathbb R_{\geq 0}\rightarrow\mathbb R_{\geq 0}\}_{s,t\in\mathbb Z_{\geq 0}}$ as follows: $\lambda_{s,t}(x):=\min\{(f(\rqt))^2\,|\,\rqt\in(\chi\circ\Psi)^{-1}(x)\cap(\closu{B_{s+1}\setminus B_{s-1}}\times\closu{B'_{t+1}\setminus B'_{t-1}})\}$ or $1$ if $(\chi\circ\Psi)^{-1}(x)\cap(\closu{B_{s+1}\setminus B_{s-1}}\times\closu{B'_{t+1}\setminus B'_{t-1}})=\emptyset$. These functions give minima of $f^2$ on compact pieces of fibers of $\chi$.

To estimate the rate of vanishing of $f^2$ at $\Psi^{-1}(V)$ we need to look at lower bounds of $\lambda_{s,t}$'s over intervals. We define a sequence of functions $\mathbb R_{\geq 0}\rightarrow\mathbb R_{\geq 0}$ as follows: $\forall l\in\mathbb N$ $\alpha_{s,t}(\frac{1}{l}):=\min\{\lambda_{s,t}(x)\,|\,x\in[\frac{1}{l+1},\infty)\}$,\footnote{Each $\alpha_{s,t}(\frac{1}{l})$ is well defined and positive, because $((\chi\circ\Psi)^{-1}[\frac{1}{l+1},\infty))\cap(\closu{B_{s+1}\setminus B_{s-1}}\times\closu{B'_{t+1}\setminus B'_{t-1}})$ is a compact subset of $\mathbb R^{m+n}\setminus\Psi^{-1}(V)$.} $\alpha_{s,t}$ is linear on each $[\frac{1}{l+1},\frac{1}{l}]$, $\alpha_{s,t}|_{[1,\infty)}=\alpha_{s,t}(1)$ and $\alpha_{s,t}(0):=0$. Each $\alpha_{s,t}$ is continuous, weakly monotonically increasing and $\vset{\alpha_{s,t}}{x}=\{0\}$. 

These functions, composed with $\chi$, estimate from below the rate of vanishing of $f^2$ at compact pieces of $\Psi^{-1}(V)$: let $W:=(\chi\circ\Psi)^{-1}([0,1))$, it is an open neighbourhood of $\Psi^{-1}(V)$, s.t.\@ $\forall s,t\in\mathbb Z_{\geq 0}$ $(f(\rqt))^2\geq\alpha_{s,t}(\chi(\Psi(\rqt)))$ for all $\rqt\in W\cap(\closu{B_{s+1}\setminus B_{s-1}}\times\closu{B'_{t+1}\setminus B'_{t-1}})$. Indeed, let $\rqt\in W$, by construction $\chi(\Psi(\rqt))\in [\frac{1}{l+1},\frac{1}{l}]$ for some $l\geq 1$, and if $\rqt\in\closu{B_{s+1}\setminus B_{s-1}}\times\closu{B'_{t+1}\setminus B'_{t-1}}$, then $\alpha_{s,t}(\chi(\Psi(\rqt)))\leq\alpha_{s,t}(\frac{1}{l})=\min(\lambda_{s,t}(x)\,|\,x\in[\frac{1}{l+1},\infty))\leq(f(\rqt))^2$, since $\rqt\in(\chi\circ\Psi)^{-1}([\frac{1}{l+1},\infty))\cap(\closu{B_{s+1}\setminus B_{s-1}}\times\closu{B'_{t+1}\setminus B'_{t-1}})$.

\smallskip

We can glue $\alpha_{s,t}$'s into a sequence of functions on $\mathbb R^m$: let $\{\epsilon_s\}_{s\in\mathbb Z_{\geq 0}}$ be a smooth partition of unity on $\mathbb R^m$, s.t.\@ $\{\rpt\,|\,\epsilon_s(\rpt)>0\}\subseteq B_{s+\frac{2}{3}}\setminus B_{s-\frac{2}{3}}$. For each $t\geq 0$ $\underset{s\geq 0}\sum\epsilon_s(\rpt)\alpha_{s,t}(\chi(\rpt))$ is a continuous non-negative function on $\mathbb R^m$ that vanishes exactly on $V$. Moreover, for each $\rqt\in W\cap(\mathbb R^m\times\closu{B'_{t+1}\setminus B'_{t-1}})$ we have $(f(\rqt))^2\geq\underset{s\geq 0}\sum\epsilon_s(\Psi(\rqt))\alpha_{s,t}(\chi(\Psi(\rqt)))$.

For each $s\in\mathbb Z_{\geq 0}$ instead of a sequence $\{\alpha_{s,t}\}_{t\in\mathbb Z_{\geq 0}}$ of continuous functions we would like to find one smooth function, that vanishes at $0$ faster than any power of each one of $\alpha_{s,t}$'s. A very old theorem by du Bois-Reymond (\cite{Ha1910}, page 10) lets us do it. Here is a variation of this result.

\begin{proposition}\label{InfinitesimalBoisReymond} Let $\{\alpha_t\}_{t\in\mathbb Z_{\geq 0}}$ be a sequence of continuous weakly monotonically increasing functions $\mathbb R_{\geq 0}\rightarrow\mathbb R_{\geq 0}$, s.t.\@ $\forall t$ $\vset{\alpha_t}{x}=\{0\}$. There is non-negative $\alpha\in\cinfty(\mathbb R)$, s.t.\@ $\vset{\alpha}{x}=\mathbb R_{\leq 0}$ and $\forall t,k\in\mathbb Z_{\geq 0}$ $\underset{x\rightarrow 0}\lim\,\frac{\alpha(x)}{(\alpha_t(x))^k}=0$.
\end{proposition}
\begin{proof} It is clear that, given another sequence $\{\alpha'_t\}$, s.t.\@ $\forall t$ $\alpha'_t\leq\alpha_t$ on $[0,1]$, if we prove the statement for $\{\alpha'_t\}$, it would follow for $\{\alpha_t\}$ as well. Thus, dividing each $\alpha_t$ by $\max(2,\underset{x\in[0,1]}\max\alpha(x))$, we can assume that $\forall t\geq 0\,\forall x\in[0,1]$ $\alpha_t(x)< 1$. In turn, putting $\alpha'_t:=\alpha_0\cdot\ldots\cdot\alpha_t$, we can assume that $\forall t\geq 0\,\forall x\in(0,1]$ $\alpha_{t+1}(x)<\alpha_t(x)$ and $\underset{x\rightarrow 0}\lim\frac{\alpha_{t+1}(x)}{\alpha_t(x)}=0$. 

Now we define a piece-wise linear $\beta\colon\mathbb R_{\geq 0}\rightarrow\mathbb R_{\geq 0}$ as follows: $\forall t\in\mathbb N$ $\beta(\frac{1}{t}):=\alpha_t(\frac{1}{t})$, on $[\frac{1}{t+1},\frac{1}{t}]$ $\beta$ is linear, $\beta(0):=0$ and $\beta|_{[1,\infty)}=\beta(1)$. It is clear that $\beta$ is continuous, and since $\alpha_t$'s are assumed to be increasing and $\forall t$ $\alpha_{t+1}<\alpha_t$ on $[0,1]$, $\beta$ is weakly monotonically increasing as well.

By construction $\forall t\in\mathbb N$ the subset of $[\frac{1}{t+1},\frac{1}{t}]$ where $\beta$ and $\alpha_t$ agree is not empty, let $x_t$ be the smallest element of this subset, since $\alpha_{t+1}(\frac{1}{t+1})<\alpha_t(\frac{1}{t+1})$ it is clear that $x_t>\frac{1}{t+1}$. We define $\gamma\colon\mathbb R_{\geq 0}\rightarrow\mathbb R_{\geq 0}$ as follows: $\forall t$ $\gamma|_{[\frac{1}{t+1},x_t]}:=\beta$, $\gamma|_{[x_t,\frac{1}{t}]}:=\alpha_t$, $\gamma|_{[1,\infty)}=\gamma(1)$, $\gamma(0):=0$. It is easy to see that $\gamma$ is continuous weakly monotonically increasing and $\vset{\gamma}{x}=\{0\}$. Moreover, $\forall t\in\mathbb N,\,\forall x\in(0,\frac{1}{t+1}]$ we have $\gamma(x)<\alpha_t(x)$. This is true because $\forall t'> t$ $\gamma|_{[\frac{1}{t'+1},\frac{1}{t'}]}\leq\alpha_{t'}|_{[\frac{1}{t'+1},\frac{1}{t'}]}<\alpha_{t}|_{[\frac{1}{t+1},\frac{1}{t}]}$.

Since $\gamma$ is continuous and weakly monotonically increasing, we can find $\phi\in\cinfty(\mathbb R)$, s.t.\@ $\vset{\phi}{x}=\mathbb R_{\leq 0}$ and $\phi|_{\mathbb R_{\geq 0}}\leq\gamma$, e.g.\@ we can extend $\gamma$ by $0$ to all of $\mathbb R$ and apply the following trivial lemma to $\gamma(-x)$.

\begin{lemma}\label{SmoothAppro} Let $\gamma\colon\mathbb R\rightarrow\mathbb R$ be continuous and weakly monotonically decreasing, and s.t.\@ $\{x\in\mathbb R\,|\,\gamma(x)\neq 0\}=\{x<0\}$. Let $\epsilon\in\cinfty(\mathbb R)$, s.t.\@ $0\leq \epsilon(x)\leq 1$ and $\suppop{\epsilon}=[-1,0]$. Let $\convo{\epsilon}{\gamma}$ be the convolution. It is smooth, weakly monotonically decreasing, $0\leq (\convo{\epsilon}{\gamma})(x)\leq \gamma(x)$ for all $x$, and $\{x\in\mathbb R\,|\,(\convo{\epsilon}{\gamma})(x)\neq 0\}=\{x<0\}$.\end{lemma}
\begin{proof} By definition $(\convo{\epsilon}{\gamma})(x)=\int_\mathbb R\epsilon(x-y)\gamma(y)d y$. For any $y\geq 0$ we have $\gamma(y)=0$, and, if $x\geq 0$, for any $y<0$ we have $\epsilon(x-y)=0$. Therefore for any $x\geq 0$ we have $\epsilon(x-y)\gamma(y)=0$, and hence $(\convo{\epsilon}{\gamma})(x)=0$ for $x\geq 0$.

If $x<0$, the function $\epsilon(x-y)\gamma(y)$ is not constantly $0$, as a function of $y$. Indeed, if $y<0$ and simultaneously $x<y<x+1$, then $\gamma(y)>0$ and simultaneously $\epsilon(x-y)>0$. Therefore, since both $\gamma$ and $\epsilon$ are non-negative functions, $(\convo{\epsilon}{\gamma})(x)> 0$ for all $x<0$. 

For any $x<0$ we have $\epsilon(x-y)\gamma(y)\leq\gamma(x)$, $\forall y\in\mathbb R$. Indeed, $\epsilon(x-y)\leq 1$ and hence $\epsilon(x-y)\gamma(y)\leq\gamma(y)$, and since $\gamma$ is monotonically decreasing, we obtain the inequality for $y\geq x$. For $y<x$ we have $x-y>0$, and hence $\epsilon(x-y)=0$. Thus the inequality holds for all $y$. Since $\suppop{\epsilon}=[-1,0]$, support of $\epsilon(x-y)f(y)$ as a function of $y$ is contained in an interval of length $1$. Altogether we conclude that $(\convo{\epsilon}{\gamma})(x)\leq\gamma(x)\cdot 1=\gamma(x)$.\end{proof}%

Finally we define $\alpha:=e^{-\frac{1}{\phi^2}}$. Since $\forall k\in\mathbb N$ $\underset{x\rightarrow 0}\lim\frac{\alpha(x)}{(\phi(x))^k}=0$, $\phi\leq\gamma$, and on $[0,\frac{1}{t+1}]$ $\gamma<\alpha_t$, it is clear that $\forall t,k\in\mathbb Z_{\geq 0}$ $\underset{x\rightarrow 0}\lim\frac{\alpha(x)}{(\alpha_t(x))^k}=0$.\end{proof}%

Applying Proposition \ref{InfinitesimalBoisReymond} to $\{\alpha_{s,t}\}_{t\in\mathbb Z_{\geq 0}}$ for each $s\in\mathbb Z_{\geq 0}$, we obtain a sequence $\{\alpha_s\}_{s\in\mathbb Z_{\geq 0}}$ in $\cinfty(\mathbb R)$. Let $h:=\underset{s\geq 0}\sum\epsilon_s\cdot(\alpha_s\circ\chi)$. We claim that $f^2$ divides $h\circ\Psi$ in $\cinfty(\mathbb R^{m+n})$. Indeed, the ratio $\frac{h\circ\Psi}{f^2}$ is a smooth function on $\mathbb R^{m+n}\setminus\Psi^{-1}(V)$, and we extend it by $0$ to all of $\mathbb R^{m+n}$. As the following obvious lemma shows, to prove that $\frac{h\circ\Psi}{f^2}$ is smooth everywhere, it is enough to show that $\frac{h\circ\Psi}{f^{2k}}$ is continuous for all $k\in\mathbb N$, when extended by $0$ to $\Psi^{-1}(V)$.

\begin{lemma}\label{SmoothnessThroughContinuity} Let $\mu,\nu\in\cinfty(\mathbb R^l)$, $l\geq 0$, $k\in\mathbb Z_{\geq 1}$, and let $g_k(\rpt):=\frac{\mu(\rpt)}{\nu(\rpt)^k}$, if $\nu(\rpt)\neq 0$ and $0$ otherwise. If $\forall k\in\mathbb Z_{\geq 1}$ $g_k$ is continuous on all of $\mathbb R^l$, then $\forall k\in\mathbb Z_{\geq 1}$ $g_k\in\cinfty(\mathbb R^l)$, having $0$-jet at every point of $\{\nu=0\}\subseteq\mathbb R^l$.\end{lemma}
\begin{proof} Let $U:=\{\rpt\in\mathbb R^l\,|\,\nu(\rpt)\neq 0\}$, clearly $g_k$ is smooth on $U$ for each $k$. Suppose we have proved that for some $t\in\mathbb Z_{\geq 0}$ $g_k\in\mathcal C^t(\mathbb R^l)$ for all $k\in\mathbb Z_{\geq 1}$. Then from $g_k=\nu g_{k+1}$ it follows that $g_k\in\mathcal C^{t+1}(\mathbb R^n)$.\end{proof}%

We need to check continuity of $\frac{h\circ\Psi}{f^{2k}}$ at each $\rqt_0\in\Psi^{-1}(V)$. There are $s,t\in\mathbb Z_{\geq 0}$ s.t.\@ $W\cap(\closu{B_{s+1}\setminus B_{s-1}}\times\closu{B'_{t+1}\setminus B'_{t-1}})$ contains an open neighbourhood $U\ni\rqt_0$. For each $\rqt\in U$ we have $(f(\rqt))^2\geq\alpha_{s,t}(\chi(\Psi(\rqt)))$. Therefore $\forall k\in\mathbb N$ $0\leq\underset{\rqt\rightarrow\rqt_0}\lim\frac{\alpha_s(\chi(\Psi(\rqt)))}{(f(\rqt))^{2k}}\leq\underset{\rqt\rightarrow\rqt_0}\lim\frac{\alpha_s(\chi(\Psi(\rqt)))}{(\alpha_{s,t}(\chi(\Psi(\rqt))))^{k}}=0$. The same argument applies to the other possible summand of $h$, that is non-zero arbitrarily close to $\Psi(\rqt_0)$.

\section{Proof of Proposition \ref{GermsThroughNeighbourhoods}}\label{ProofGermsThroughNeighbourhoods}

Let $\cring$ be finitely generated and point-determined, and suppose $\exists a\notin\zerogermin{\cspace}{\cspace'}$ s.t.\@ $a\in\underset{\ideal\in\rinfinil{\cspace}{\cspace'}}\bigcap\ideal\leq\kernel{\phi}$. We would like to arrive at a contradiction.

{\bf 1.} Choose a surjective $\alpha\colon\cinfty(\mathbb R^n)\rightarrow\cring$, let $V\subseteq\mathbb R^n$ consist of common zeroes of $\kernel{\alpha}$. We choose generators $\{a_i\}_{i\in I}$ of $\kernel{\phi}$ as an ideal, and $\forall i\in I$ some $g_i\in\alpha^{-1}(a_i)$, let $f\in\alpha^{-1}(a)$. We define closed subsets of $\mathbb R^n$
	\begin{equation*}V_i:=\{v\in V\,|\,f(v)=0,\,g_i(v)= 0,\,\rgerm{(f|_V)}{v}\neq 0\}, \quad i\in I,\end{equation*} 
where $\rgerm{(f|_V)}{v}$ is the germ at $v$ of the restriction of $f$ to $V$.

\smallskip

{\bf 2.} We claim that $\forall i$ $V_i\neq\emptyset$. Indeed, if $V\cap\vset{f}{\rpt}\cap\vset{g_i}{\rpt}=\emptyset$, then $a^2+a_i^2\in\cring$ is invertible, implying $\igen{a}+\kernel{\phi}=\cring$, and then $\kernel{\phi}=\cring$ , since $a\in\kernel{\phi}$. This would contradict $a\notin\zerogermin{\cspace}{\cspace'}$. 

If $(f|_V)_v=0$ $\forall v\in V\cap\vset{f}{\rpt}\cap\vset{g_i}{\rpt}$, we can choose an open $U\subseteq\mathbb R^n$, s.t.\@ $f|_{U\cap V}=0$ and $V\cap\vset{g_i}{\rpt}\cap\vset{f}{\rpt}\subseteq U$. Choose an $h\in\cinfty(\mathbb R^n)$, s.t.\@ $h|_{\mathbb R^n\setminus U}=0$ and $h|_{U}>0$, we see that $f h\in\kernel{\alpha}$ and $h$ is invertible modulo $\igen{f,g_i}+\kernel{\alpha}$. Since $f,g_i\in\kernel{\phi\circ\alpha}$, $h$ is also invertible modulo $\kernel{\phi\circ\alpha}$, contradicting $a\notin\zerogermin{\cspace}{\cspace'}$.
 
 \smallskip
 
{\bf 3.} For each $i\in I$ we would like to find a $\overline{g}_i\in\cinfty(\mathbb R^n)$, such that $\overline{g}_i$ vanishes exactly where $g_i$ vanishes, and the upper bound of $\overline{g}_i$ decays at $V_i$ along $V$ faster than that of $f$. To make this precise we need the following definition. We denote by $\czero(\mathbb R^n)$ be the set of continuous functions.

\begin{definition}\label{UpperBoundVanishing} Let $f\in\czero(\mathbb R^n)$, let $\emptyset\neq W\subseteq V\subseteq\mathbb R^n$ be closed subsets. {\it The order of upper bound vanishing of $f$ at $W$ along $V$} is the sequence of (not necessarily continuous) functions $\{\upperrate[V]{W}{f}{k}\colon W\rightarrow\mathbb R_{\geq 0}\}_{k\in\mathbb N}$, where
	\begin{equation*}\forall w\in W,\quad\upperrate[V]{W}{f}{k}(w):=\underset{\rpt\in\ball{w}{\frac{1}{k}}\cap V}\max\; |f(\rpt)|,\end{equation*}
and $\ball{w}{\frac{1}{k}}\subseteq\mathbb R^n$ is the closed ball of radius $\frac{1}{k}$ centered at $w$.

Suppose $\forall k\in\mathbb N$ $\upperrate[V]{W}{f}{k}$ is no-where vanishing on $W$. We say that {\it $g\in\czero(\mathbb R^n)$ decays at $W$ along $V$ faster than $f$} and write $\decay{W,V}{g}{f}$, if
	\begin{equation*}\forall w\in W,\quad\frac{\upperrate[V]{W}{g}{k}(w)}{\upperrate[V]{W}{f}{k}(w)}\underset{k\rightarrow\infty}\longrightarrow 0.\end{equation*}
\end{definition}

The distance function $\distance{-}{W}\colon\mathbb R^n\rightarrow\mathbb R_{\geq 0}$ is continuous and vanishes exactly on $W$. The following Lemma implies that we can always find $\chi$ in $\cinfty(\mathbb R^n)$, vanishing exactly on $W$ and decaying there faster than $\distance{-}{W}$.

\begin{lemma}\label{SmoothingDistance} Let $\chi\in\cinfty(\mathbb R^n)$, $W:=\vset{\chi}{\rpt}$. Suppose that partial derivatives of $\chi$ vanish on $W$. Then $|\chi|\leq\distance{-}{W}$ on an open $U\supseteq W$.\end{lemma}
\begin{proof} Suppose not. Then $\exists\{\rqt_k\}_{k\in\mathbb N}\subseteq\mathbb R^n\setminus W$ converging to $\rpt\in W$, s.t.\@ $\forall k$ $|\chi(\rqt_k)|>\distance{\rqt_k}{W}$. Applying the mean value theorem we get a sequence $\{\rqt'_k\}\underset{k\rightarrow\infty}\longrightarrow\rpt$, s.t.\@ at each $\rqt'_k$ $\chi$ has a directional derivative with absolute value $>1$. This implies that at least one partial derivative of $\chi$ has absolute value $>\frac{1}{n}$ on a sub-sequence of $\{\rqt'_k\}$, contradicting the assumptions.\end{proof}%

{\bf 4.} By definition $\forall i\in I\,\forall v\in V_i$ $\rgerm{(f|_V)}{v}\neq 0$. Therefore $\forall k\in\mathbb N,\forall v\in V_i$ $\upperrate[V]{V_i}{f}{k}(v)\neq 0$, and we can ask for $\overline{g}_i\in\cinfty(\mathbb R^n)$, s.t.\@ $\vset{\overline{g}_i}{\rpt}=\vset{g_i}{\rpt}$ and $\decay{V_i,V}{\overline{g}_i}{f}$. The following lemma is almost obvious.

\begin{lemma}\label{DecayOnCompact} In the situation of Definition \ref{UpperBoundVanishing} suppose $W$ is compact. Suppose that $\forall w\in W$ $\rgerm{(f|_V)}{w}\neq 0$. Then $\forall k\in\mathbb N$ $\underset{w\in W}{\inf}\;\upperrate[V]{W}{f}{k}(w)>0$.\end{lemma}
\begin{proof} Let $k\in\mathbb N$, s.t.\@ the subset $\upperrate[V]{W}{f}{k}(W)\subseteq\mathbb R_{\geq 0}$ is infinite. Consider a strictly decreasing sequence $\{b_i\}_{i\in\mathbb N}\subseteq\upperrate[V]{W}{f}{k}(W)$. Choosing $\{w_i\}_{i\in\mathbb N}\subseteq W$, s.t.\@ $\upperrate[V]{W}{f}{k}(w_i)=b_i$, and using compactness of $W$, we can find a converging sub-sequence of $\{w_i\}_{i\in\mathbb N}$. We can assume that the sequence $\{w_i\}_{i\in\mathbb N}$ itself converges to $w\in W$. It is clear that $\underset{\rpt'\in B_{w,\frac{1}{2k}}\cap V}\max\;|f(\rpt')|\leq\underset{i\rightarrow\infty}\lim\,b_i$. Therefore, since $\rgerm{f}{w}\neq 0$, it must be that $\underset{i\rightarrow\infty}\lim\,b_i>0$.\end{proof}%

\begin{lemma}\label{SmoothUpperBound} In the situation of Definition \ref{UpperBoundVanishing} suppose $\forall w\in W$ $\rgerm{(f|_V)}{w}\neq 0$. Then $\exists g\in\cinfty(\mathbb R^n)$, s.t.\@ $\vset{g}{\rpt}=W$ and $\decay{W,V}{g}{f}$.\end{lemma}
\begin{proof} Given $j\in\mathbb N$ let $W_j:=W\cap(B_{0,j+1}\setminus\overset{\circ}{B}_{0,j-1})$, where $\overset{\circ}{B}_{0,j-1}$ is the open ball of radius $j-1$ centered at the origin. Clearly $W=\underset{j\in\mathbb N}\bigcup\,W_j$. For every $j\in\mathbb N$ we define a piece-wise linear function $h_j\colon\mathbb R_{\geq 0}\rightarrow\mathbb R_{\geq 0}$ as follows
	\begin{equation*}\forall k\in\mathbb N,\quad h_j(\frac{1}{k})=\underset{w\in W_j}\inf\,\upperrate[V]{W}{f}{k}(w)
	\text{ and }h_j(d)=h_j(1)\text{ for }d>1.\end{equation*}
It is easy to see that each $h_j$ is weakly monotonically increasing. From Lemma \ref{DecayOnCompact} we see that $h_j|_{\mathbb R_{>0}}\neq 0$. Applying Lemma \ref{InfinitesimalBoisReymond} to $\{x h_j\}_{j\in\mathbb N}$, we get a weakly monotonically increasing $h\in\cinfty(\mathbb R)$ vanishing exactly on $\mathbb R_{\leq 0}$ and decaying at $0$ faster than each $h_j$. Choose $\chi\in\cinfty(\mathbb R^n)$ that vanishes exactly on $W$, then so does $e^{-\chi^{-2}}$ together with its derivatives, thus (Lemma \ref{SmoothingDistance}) $e^{-\chi^{-2}}\leq\distance{-}{W}$ in a neighbourhood of $W$. We define $g:=h\circ e^{-\chi^{-2}}$.

Let $w\in W$, then $\exists j\in\mathbb N$ s.t.\@ $w\in W_j$. For $k\in\mathbb N$ large enough we have $\rpt\in\ball{w}{\frac{1}{k}}\Rightarrow\distance{\rpt}{W}\geq e^{-\chi^{-2}}(\rpt)$. Then
	\begin{equation*}\frac{\upperrate[V]{W}{g}{k}(w)}{\upperrate[V]{W}{f}{k}(w)}\leq
	\frac{\upperrate[V]{W}{g}{k}(w)}{h_j(\frac{1}{k})}\leq\frac{h(\frac{1}{k})}{h_j(\frac{1}{k})}\underset{k\rightarrow\infty}{\longrightarrow}0.\end{equation*}
\end{proof}%

Applying Lemma \ref{SmoothUpperBound} to $f$, $V_i$ and $V$ we can find $\widetilde{g}_i\in\cinfty(\mathbb R^n)$ s.t.\@ $\vset{\overline{g}_i}{\rpt}=V_i$ and $\decay{V_i,V}{\widetilde{g}_i}{f}$. Then $\overline{g}_i:=g_i\widetilde{g}_i$.

\smallskip

{\bf 5.} Consider the ideal $\ideal:=\kernel{\alpha}+\underset{i\in I}\sum\igen{\overline{g}_i}$. Clearly $\iradical{\ideal}=\iradical{\kernel{\phi\circ\alpha}}$, but $f\notin\ideal\supseteq\underset{\ideal\in\rinfinil{\cspace}{\cspace'}}\bigcap\ideal$. Indeed, suppose $f=c+\underset{1\leq j\leq m}\sum c_j\overline{g}_j$, where $c\in\kernel{\alpha}$, $c_j\in\cinfty(\mathbb R^n)$. Since $\forall j$ $f|_V$ decays at $V_j$ faster than $\overline{g}_j$, and $c\equiv 0$ on $V$, the following lemma implies that $\underset{1\leq j\leq m}\bigcap V_j=\emptyset$. 

\begin{lemma} Let $f\in\czero(\mathbb R^n)$, $\emptyset\neq W\subseteq V\subseteq\mathbb R^n$ closed. Suppose $\forall w\in W$ $\rgerm{(f|_V)}{w}\neq 0$, and let $\{g_i\}_{1\leq i\leq k}\subseteq\czero(\mathbb R^n)$, s.t. $\forall i$ $\decay{W,V}{g_i}{f}$. Then $\forall h_i\in\czero(\mathbb R^n)$
	\begin{equation*}\decay{W,V}{\underset{1\leq i\leq k}\sum h_i g_i}{f}.\end{equation*}\end{lemma}
\begin{proof} Clearly for any $k\in\mathbb Z_{\geq 1}$ and any $w\in W$
	\begin{equation*}\upperrate[V]{W}{\underset{1\leq i\leq k}\sum h_i g_i}{k}(w)\leq
	\underset{1\leq i\leq k}\sum\upperrate[V]{W}{h_i}{k}(w)\upperrate[V]{W}{g_i}{k}(w).\end{equation*}
Dividing by $\upperrate[V]{W}{f}{k}$ and recalling that each $\upperrate[V]{W}{h_i}{k}(w)$ is a weakly monotonically decreasing function of $k$ we are done.\end{proof}%
Now $\underset{1\leq j\leq m}\bigcap V_j=\emptyset$ means $f|_V$ has $0$ germ at $V\cap \underset{1\leq j\leq m}\bigcap\vset{g_j}{\rpt}$, which in turn implies existence of $h\in\cinfty(\mathbb R^n)$, s.t.\@ $f h\in\kernel{\alpha}$ and $h$ is invertible modulo $\kernel{\alpha}+\underset{1\leq j\leq m}\sum\igen{g_j}$, contradicting our choice of $a$.

\end{document}